\documentclass[12pt,leqno]{amsart}
\usepackage[latin9]{inputenc}
\usepackage{verbatim}
\usepackage{amsthm}
\usepackage{amstext}
\usepackage{amssymb}
\usepackage[dvipdfmx]{graphicx}

\makeatletter
\numberwithin{equation}{section}
\numberwithin{figure}{section}

\usepackage{amsthm}\usepackage{amscd}\usepackage{bm}
\theoremstyle{plain}
\newtheorem{theorem}{Theorem}[section]
\newtheorem{lemma}[theorem]{Lemma}

\newtheorem{corollary}[theorem]{Corollary}
\newtheorem{proposition}[theorem]{Proposition}

\newtheorem{sublemma}[theorem]{Sublemma}
\theoremstyle{definition}

\newtheorem{remark}[theorem]{Remark}

\newtheorem{notation}[theorem]{Notation}

\newtheorem{problem}[theorem]{Problem}
\newtheorem{condition}[theorem]{Condition}
\newcommand{\mdim}{\mathrm{mdim}}
\newcommand{\diam}{\mathrm{diam}}
\newcommand{\widim}{\mathrm{Widim}}
\newcommand{\dist}{\mathrm{dist}}
\newcommand{\supp}{\mathrm{supp}}
\newcommand{\norm}[1]{\left|\!\left|#1\right|\!\right|}

\makeatother

\begin{document}

\title[Embedding minimal systems into Hilbert cubes]{Embedding minimal dynamical systems into Hilbert cubes}

\author{Yonatan Gutman, Masaki Tsukamoto}



\date{\today}

\thanks{Y.G was partially supported by the Marie Curie grant PCIG12-GA-2012-334564
and by the National Science Center (Poland) grant 2013/08/A/ST1/00275.
M.T. was supported by Grant-in-Aid for Young Scientists (B)
25870334 from JSPS}

\maketitle

\begin{abstract}
We study the problem of embedding minimal dynamical systems into the shift action on the Hilbert cube $\left([0,1]^N\right)^{\mathbb{Z}}$.
This problem is intimately related to the theory of mean dimension, which counts the averaged number
of parameters of dynamical systems.
Lindenstrauss proved that minimal systems of mean dimension less than $N/36$ can be embedded into 
$\left([0,1]^N\right)^{\mathbb{Z}}$, and he proposed the problem of finding the optimal value of the mean dimension for the 
embedding.
We solve this problem by proving that minimal systems of mean dimension less than $N/2$ can be embedded into 
$\left([0,1]^N\right)^{\mathbb{Z}}$.
The value $N/2$ is optimal.
The proof uses Fourier and complex analysis.
\end{abstract}

\section{Introduction}  \label{section: introduction}

\subsection{Embedding into Hilbert cubes}  \label{subsection: embedding into Hilbert cubes}

A tuple $(X,T)$ is called a \textbf{dynamical system} if 
$X$ is a compact metric space and $T$ is a homeomorphism of $X$.
Basic examples for us are \textbf{the shifts on the Hilbert cubes}:
Let $N$ be a natural number and consider the infinite product 
\[ \left([0,1]^{N}\right)^{\mathbb{Z}} = \cdots\times  [0,1]^N \times [0,1]^N\times [0,1]^N\times \cdots.\]
We define the shift $\sigma$ on it by 
\[ \sigma\left((x_n)_{n\in \mathbb{Z}}\right) =  (x_{n+1})_{n\in \mathbb{Z}},   \quad  \text{where $x_n\in [0,1]^N$}.  \]
$\left(\left([0,1]^N\right)^{\mathbb{Z}}, \sigma\right)$ is a dynamical system.
We study the problem of embedding arbitrary dynamical systems into $\left(\left([0,1]^N\right)^{\mathbb{Z}}, \sigma\right)$.
More formally, we study 
\begin{problem}[Embedding Problem]
Let $(X,T)$ be a dynamical system. Decide whether there exists a topological embedding 
\[  f: X\to \left([0,1]^N\right)^{\mathbb{Z}} \]
satisfying $f\circ T = \sigma \circ f$.
Such a map $f$ is called an embedding of a dynamical system.
\end{problem}

For example, consider an irrational rotation 
\[ X= \mathbb{R}/\mathbb{Z}, \quad T(x) = x+\alpha, \quad \alpha \in \mathbb{R}\setminus \mathbb{Q}.\]
Then the map 
\[ \mathbb{R}/\mathbb{Z} \to [0,1]^\mathbb{Z}, \quad 
   x\mapsto   \left(\frac{1+\cos(2\pi(x+n\alpha))}{2}\right)_{n\in \mathbb{Z}} \]
is an embedding of the irrational rotation $(\mathbb{R}/\mathbb{Z}, T)$.   
This example is very simple.
But in general the problem is much more involved and still not fully understood.
We quickly review the history of the problem before explaining the main result.

An obvious obstruction for the embedding comes from periodic points; 
if $(X,T)$ has too many periodic points then it cannot be embedded into the shift on $\left([0,1]^N\right)^{\mathbb{Z}}$.
For example, the shift on $\left([0,1]^2\right)^{\mathbb{Z}}$ cannot be embedded into the shift on 
$[0,1]^{\mathbb{Z}}$ because the fixed points set of the former is homeomorphic to 
 $[0,1]^2$, which cannot be embedded into $[0,1]$.
Somewhat surprisingly, Jaworski \cite{Jaworski} proved that periodic points are the \textit{only} obstruction if $X$ is finite dimensional:

\begin{theorem}[Jaworski, 1974]
If $(X,T)$ is a finite dimensional system having no periodic points, then we can embed it into the 
shift on $[0,1]^{\mathbb{Z}}$.
\end{theorem}

The first named author \cite{Gutman 2} extended 
this result to the case of finite dimensional systems having \textit{reasonable} amount of periodic points.
The embedding problem for finite dimensional systems is fairly well understood now.
Therefore the main targets of our study are infinite dimensional systems.
But completely general infinite dimensional systems are still beyond our present technologies.
We have to consider some restrictions on our systems.

Probably the most fundamental dynamical systems are minimal systems.
A system $(X,T)$ is said to be \textbf{minimal} if for every $x\in X$ the orbit 
\[  \dots,T^{-3}x, T^{-2}x, T^{-1}x, x, Tx, T^2x, T^3x, \dots \] 
is dense in $X$.
Minimal systems have no periodic points unless they are finite.
(Finite systems are trivial cases.)
So there is no ``periodic points obstruction''.
Auslander \cite[p.193]{Auslander} asked whether we can embed every minimal system into the shift on 
$[0,1]^{\mathbb{Z}}$.
In other words, he asked whether there is another obstruction different from periodic points.
This problem remained open for more than 10 years.

Lindenstrauss--Weiss \cite{Lindenstrauss--Weiss} solved Auslander's problem
by using the theory of mean dimension.
Mean dimension is a topological invariant of dynamical systems introduced by 
Gromov \cite{Gromov}.
It counts the \textit{number of parameters of systems per second} like topological entropy counts the \textit{number of bits
per second for describing dynamical systems}.
We review the definition in Section \ref{subsection: review of mean dimension}.
The mean dimension of the shift on $\left([0,1]^N\right)^{\mathbb{Z}}$ is equal to $N$.
This is a rigorous statement of the intuitive idea that the system $\left([0,1]^N\right)^{\mathbb{Z}}$ has $N$ parameters per second.
If a system $(X,T)$ is embeddable into the shift on $\left([0,1]^N\right)^{\mathbb{Z}}$ then its mean dimension 
(denoted by $\mdim(X,T)$) is less than or equal to $N$.
Lindenstrauss--Weiss \cite[Proposition 3.5]{Lindenstrauss--Weiss} constructed a minimal system of mean dimension strictly greater than one.
So this system cannot be embedded into the shift on $[0,1]^{\mathbb{Z}}$ although it is minimal.

It is a big surprise that a partial converse holds (\cite[Theorem 5.1]{Lindenstrauss}):
\begin{theorem}[Lindenstrauss, 1999]
If $(X,T)$ is a minimal system with 
\[ \mdim(X,T) < \frac{N}{36} ,\]
then we can embed it into the shift on $\left([0,1]^N\right)^{\mathbb{Z}}$.
\end{theorem}

This is a wonderful theorem.
But the number $N/36$ looks artificial.
Quoting Lindenstrauss \cite[p. 229]{Lindenstrauss} with a slight change of notations:
\begin{quote}
   \textit{Another nice question that remains open is what is the largest constant $c$ such that $\mdim(X,T)< c N$
   implies that $(X,T)$ can be embedded in $\left(\left([0,1]^N\right)^{\mathbb{Z}}, \mathrm{shift}\right)$?
   The bound we get is that $c\geq 1/36$.}
\end{quote}   
We solve this problem. The answer is $c=1/2$.
Namely 

\begin{theorem}[Main Theorem]  \label{thm: main theorem}
If $(X,T)$ is a minimal system with 
\[ \mdim(X,T) < \frac{N}{2}, \]
then we can embed it into the shift on $\left([0,1]^N\right)^{\mathbb{Z}}$.
\end{theorem}

The value $N/2$ is optimal because 
Lindentsrauss and the second named author \cite[Theorem 1.3]{Lindenstrauss--Tsukamoto} constructed a minimal system of mean dimension  
$N/2$ which cannot be embedded into the shift on $\left([0,1]^N\right)^{\mathbb{Z}}$.
The statement of Theorem \ref{thm: main theorem} also holds for extensions of 
nontrivial (i.e. infinite) minimal systems; see Corollary \ref{cor: embed extensions of nontrivial minimal system into Hilbert cubes}.
Therefore the embedding problem is now well understood for nontrivial minimal systems and their extensions.

The proof of Theorem \ref{thm: main theorem} has a fascinating feature.
The nature of the statement itself is purely abstract topological dynamics.
But crucial ingredients of the proof are \textit{Fourier analysis} and \textit{complex function theory}.
Therefore the theorem exhibits a new unexpected interaction between topological dynamics and 
classical analysis.

\subsection{Embedding via signal processing}  \label{subsection: embedding via signal processing}

Elements $x$ of $\left([0,1]^N\right)^{\mathbb{Z}}$ are \textbf{discrete signals} valued in the $N$-dimensional cube $[0,1]^N$:
\[  \dots \> x_{-3} \> x_{-2} \> x_{-1} \> x_0 \> x_1 \> x_2 \> x_3 \> \dots, \quad \text{where } x_n\in [0,1]^N. \]
(Here ``discrete'' means ``time-discrete''.)
Informally speaking, the embedding problem asks \textit{how to encode dynamical systems into discrete signals}.
Our approach in Theorem \ref{thm: main theorem} is the following:
\begin{equation*}
  \begin{CD}
   \text{A dynamical system} @>{\text{encode}}>> \text{Continuous signals}
   @>{\text{sampling}}>> \text{Discrete signals}.
  \end{CD}
\end{equation*}
First we encode a given system into \textbf{(time-)continuous signals}.
Next we convert continuous signals into discrete ones by sampling.
Continuous signals are more flexible than discrete ones 
(see Remark \ref{remark: why we use continuous signals}), and 
we can prove the sharp embedding result.

We prepare some definitions on signal analysis.
For rapidly decreasing functions $\varphi:\mathbb{R}\to \mathbb{C}$ we define 
the Fourier transforms by 
\[ \mathcal{F}(\varphi)(\xi) = \int_{-\infty}^\infty e^{-2\pi \sqrt{-1}t\xi} \varphi(t) dt , \quad 
    \overline{\mathcal{F}}(\varphi)(t) = \int_{-\infty}^\infty e^{2\pi\sqrt{-1}t\xi} \varphi(\xi)d\xi.   \]
We have $\overline{\mathcal{F}}(\mathcal{F}(\varphi)) = \mathcal{F}(\overline{\mathcal{F}}(\varphi)) = \varphi$. 
We extend $\mathcal{F}$ and $\overline{\mathcal{F}}$ to tempered distributions in the standard way
(Schwartz \cite[Chapter 7]{Schwartz}).
For example, $\mathcal{F}(1) = \boldsymbol{\delta}_0$ is the delta probability measure at the origin.
Take two real numbers $a<b$.
A bounded continuous function $\varphi:\mathbb{R}\to \mathbb{C}$ is said to be 
\textbf{band-limited in $[a,b]$} if
\[   \supp \mathcal{F}(\varphi) \subset [a,b]. \]
Here recall that $\supp \mathcal{F}(\varphi)\subset [a,b]$ means that
the pairing $\langle \mathcal{F}(\varphi), \phi\rangle$ vanishes for any rapidly decreasing function $\phi:\mathbb{R}\to \mathbb{C}$
with $\supp(\phi)\cap [a,b]=\emptyset$.
Let $V[a,b]$ be the space of bounded continuous functions $\varphi:\mathbb{R}\to \mathbb{C}$ band-limited in $[a,b]$.
This is a Banach space with respect to the $L^\infty$-norm over the line $\mathbb{R}$.

For two functions $\varphi_1, \varphi_2\in V[a,b]$ we define a distance between them by 
\begin{equation}  \label{eq: distance on V[a,b]}
    \boldsymbol{d}(\varphi_1,\varphi_2) = \sum_{n=1}^\infty 2^{-n} \norm{\varphi_1-\varphi_2}_{L^\infty([-n,n])}.
\end{equation}    
We define $B_1(V[a,b])$ as the space $\varphi\in V[a,b]$ satisfying $\norm{\varphi}_{L^\infty(\mathbb{R})} \leq 1$.
This is compact with respect to the distance $\boldsymbol{d}$; 
see Lemma \ref{lemma: compactness of B_1(V[a,b])} in 
Section \ref{subsection: review of band-limited functions}.
Throughout the paper, $B_1(V[a,b])$ is always endowed with the topology given by $\boldsymbol{d}$,
which coincides with the standard topology of tempered distributions \cite[Chapter 7, Section 4]{Schwartz}.
We define 
\[ \sigma: V[a,b]\to V[a,b], \quad  \sigma(\varphi)(t) = \varphi(t+1), \]
and consider the dynamical system $(B_1(V[a,b]), \sigma)$.
We call this system \textbf{the shift on} $B_1(V[a,b])$.
This is related to the shifts on the Hilbert cubes by the next lemma (sampling).

\begin{lemma}  \label{lemma: sampling theorem for real signals}
Let $N$ be a natural number.
Let $c>0$ and consider the space $V^\mathbb{R}[-c,c]$ of bounded continuous functions 
$\varphi:\mathbb{R}\to \mathbb{R}$ satisfying $\supp \mathcal{F}(\varphi)\subset [-c,c]$.
If $c<N/2$ then 
we can embed the system $(B_1(V^{\mathbb{R}}[-c,c]), \sigma)$ into the shift on $\left([-1,1]^N\right)^{\mathbb{Z}}$.
\end{lemma}

\begin{proof}
By a sampling theorem  (see Lemma \ref{lemma: sampling theorem} in Section \ref{subsection: review of band-limited functions}), 
the map 
\[ V^\mathbb{R}[-c,c]\to \ell^\infty\left(\frac{1}{N}\mathbb{Z}\right), \quad 
   \varphi \mapsto \varphi|_{\frac{1}{N}\mathbb{Z}} \]
 is injective.
 The above statement follows from this.
\end{proof}

\begin{remark}
The mean dimensions of the shifts on $B_1(V[a,b])$ and $B_1(V^{\mathbb{R}}[-c,c])$ are $2(b-a)$ and $2c$ respectively.
More generally, if we denote by $V(E)$ the space of bounded continuous functions in $\mathbb{R}$ 
band-limited in a compact subset $E\subset \mathbb{R}$ then the mean dimension of the shift on 
$B_1(V(E))$ is equal to $2|E|$. Here $|E|$ is the Lebesgue measure of $E$.
This fact is probably helpful for clarifying the picture.
But we don't need it for the proof of Theorem \ref{thm: main theorem}.
So we omit the detailed explanations in this paper.
\end{remark}

The next result is the continuous signal version of the main theorem.

\begin{theorem} \label{thm: embed minimal systems into continuous signals}
If $(X,T)$ is a nontrivial minimal system with 
\[  \mdim(X,T)<b-a, \]
then we can embed it into 
the shift on $B_1(V[a,b])$.
Here ``nontrivial'' means that $X$ is an infinite set.
\end{theorem} 

Theorem \ref{thm: embed minimal systems into continuous signals}
is proved in Section \ref{subsection: proof of continuous signal version of the main theorem}.

\begin{proof}[Proof of Theorem \ref{thm: main theorem}, assuming Theorem \ref{thm: embed minimal systems into continuous signals}]
If $X$ is finite, then the statement is trivial.
So we assume that $(X,T)$ is a nontrivial minimal system.
Take $0<a<b <N/2$ with $\mdim(X,T)<b-a$.
By Theorem \ref{thm: embed minimal systems into continuous signals} we can embed $(X,T)$ into the system 
$B_1(V[a,b])$, which becomes a subsystem of $B_1(V^{\mathbb{R}}[-b,b])$ by 
\[ B_1(V[a,b]) \to B_1(V^{\mathbb{R}}[-b,b]), \quad \varphi \mapsto \frac{1}{2}(\varphi + \overline{\varphi}).\]
By Lemma \ref{lemma: sampling theorem for real signals}
we can embed $B_1(V^{\mathbb{R}}[-b,b])$ into the shift on 
$\left([-1,1]^N\right)^{\mathbb{Z}} \cong \left([0,1]^N\right)^{\mathbb{Z}}$.
\end{proof}

\subsection{Open problems} \label{subsection: open problems}

The following are the most significant questions in the direction of the paper.

 \begin{itemize}
   \item \textit{Can one solve the embedding problem for general dynamical systems?}
   The case of minimal systems is fairly well understood now. 
   But we still don't have a clear picture for more general dynamical systems. 
   Lindenstrauss and the second named author \cite[Conjecture 1.2]{Lindenstrauss--Tsukamoto} conjectured that if a 
   dynamical system $(X,T)$ satisfies 
   \[  \mdim(X,T) < \frac{N}{2}, \quad \frac{\dim \{x|\, T^n x=x\}}{n} < \frac{N}{2} \quad (\forall n\geq 1), \]
   then we can embed it into the shift on $([0,1]^N)^{\mathbb{Z}}$.
   
  \item \textit{Can one generalize the result to the case of non-commutative group actions?}
  Probably it is possible to generalize the result to the case of $\mathbb{Z}^k$-actions
  by using the techniques of \cite{Gutman--Lindenstrauss--Tsukamoto} and the present paper.
  But the generalization to non-commutative groups seems to require substantially new ideas.
 \end{itemize}

\textbf{Acknowledgment.}
The authors would like to deeply thank Professor Elon Lindenstrauss.
His influence is prevailing in the paper.
The authors learned from him the most important ideas such as \textit{signal processing} and 
\textit{Voronoi diagram in one-dimension higher space}.

\section{Basic materials}  \label{section: basic materials}

We review mean dimension and band-limited functions in this section.

\subsection{Review of mean dimension}  \label{subsection: review of mean dimension}

Here we review the definition of mean dimension.
For the details, see Gromov \cite{Gromov} and Lindenstrauss--Weiss \cite{Lindenstrauss--Weiss}.

Let $(X,d)$ be a compact metric space with a continuous function $\rho:X\times X\to \mathbb{R}$.
Let $Y$ be a topological space.
For $\varepsilon>0$, a continuous map $f:X\to Y$ is called an \textbf{$\varepsilon$-embedding with respect to $\rho$} if 
it satisfies 
\[ f(x)=f(y) \Longrightarrow \rho(x, y) < \varepsilon.\]
Note that this is an open condition for $f$ in the compact-open topology.
We usually consider the case of $\rho=d$, but sometimes $\rho$ is a semi-distance 
different from $d$.

We define $\widim_\varepsilon(X,d)$ as the minimum integer $n$ such that there exist an $n$-dimensional finite simplicial complex $P$
and an $\varepsilon$-embedding $f:X\to P$ with respect to the distance $d$.
It is classically known that the topological dimension $\dim X$ is recovered by 
\[ \dim X = \lim_{\varepsilon \to 0} \widim_\varepsilon (X,d).\]

Let $T:X\to X$ be a homeomorphism.
For a natural number $N$ we define a distance $d_N$ on $X$ by 
\[   d_N(x,y) = \max_{0\leq n<N} d(T^n x, T^n y) . \]
We define the \textbf{mean dimension} of the dynamical system $(X,T)$ by 
\[  \mdim(X,T) = \lim_{\varepsilon \to 0} \left(\lim_{N\to \infty} \frac{\widim_\varepsilon (X,d_N)}{N}\right). \]
This limit exists because the function $N\mapsto \widim_\varepsilon (X, d_N)$ is subadditive.
The mean dimension is a topological invariant of the dynamical system $(X,T)$, namely, it is independent of the choice of the distance $d$.

\subsection{Review of band-limited functions}  \label{subsection: review of band-limited functions}

Here we review basic properties of band-limited functions.
All the results in this subsection are classically known.
But some of them are not very popular in general mathematical community.
So we provide self-contained proofs, assuming (hopefully) only well-known results.
For systematic treatments, see Beurling \cite[pp. 341-365]{Beurling} and Nikol'skii \cite[Chapter 3]{Nikol'skii}.

\begin{lemma}  \label{lemma: exponential type}
Let $a>0$, and 
let $f:\mathbb{C}\to \mathbb{C}$ be a holomorphic function satisfying 
\begin{equation}  \label{eq: exponential type}
  \exists C>0 \> \forall x,y\in \mathbb{R}: \> |f(x+y\sqrt{-1})| \leq C e^{2\pi a|y|}.
\end{equation}
Then it satisfies 
\[  |f(x+y\sqrt{-1})| \leq e^{2\pi a|y|} \norm{f}_{L^\infty(\mathbb{R})}.\]
Here $\norm{f}_{L^\infty(\mathbb{R})}$ is the supremum of $|f|$ over the real line.
\end{lemma}

\begin{proof}
Let $\varepsilon>0$ and set $g(z) = e^{2\pi(a+\varepsilon)z \sqrt{-1}} f(z)$.
For $y\geq 0$ 
\[ |g(x+y\sqrt{-1})| =  e^{-2\pi(a+\varepsilon)y} |f(x+y\sqrt{-1})| \leq C e^{-2\pi\varepsilon y} \to 0 \quad 
   (y\to \infty).\]
Set $L = \{y\sqrt{-1}|\, y> 0\}\subset \mathbb{C}$.  
By the Phragm\'{e}n--Lindel\"{o}f principle (\cite[Section 3.1.7]{Dym--McKean})
\[ \sup_{y\geq 0} |g(x+y\sqrt{-1})| \leq \max\left( \norm{g}_{L^\infty(\mathbb{R})}, \, \norm{g}_{L^\infty(L)}\right).\]
If $\norm{g}_{L^\infty(\mathbb{R})}  < \norm{g}_{L^\infty(L)}$ then $|g|$ attains $\sup_{y\geq 0} |g(x+y\sqrt{-1})|$ over $L$.
But this contradicts the maximum principle. Therefore 
\[  \sup_{y\geq 0} |g(x+y\sqrt{-1})| \leq \norm{g}_{L^\infty(\mathbb{R})} = \norm{f}_{L^\infty(\mathbb{R})}.\]
Thus we get $|f(x+y\sqrt{-1})|\leq e^{2\pi(a+\varepsilon)y} \norm{f}_{L^\infty(\mathbb{R})}$.
Letting $\varepsilon \to 0$, we get the desired result for $y\geq 0$.
The case $y<0$ is similar.
\end{proof}

\begin{lemma} \label{lemma: Paley--Wiener}
Let $a>0$, and let $f:\mathbb{R}\to \mathbb{C}$ be a bounded continuous function.
The following two conditions are equivalent.

\noindent 
(1) The Fourier transform $\mathcal{F}(f)$ is supported in $[-a,a]$.

\noindent 
(2) We can extend $f$ to a holomorphic function in $\mathbb{C}$ satisfying (\ref{eq: exponential type}).
\end{lemma}

\begin{proof}
If we additionally assume $f\in L^2(\mathbb{R})$, then the above equivalence is a standard theorem of Paley--Wiener
\cite[Section 3.3]{Dym--McKean}.
So the problem is how to extend the Paley--Wiener theorem to the case of bounded continuous functions.
For a more general result, see Schwartz \cite[Chapter 7, Section 8]{Schwartz}.

Let $\psi(\xi)$ be a nonnegative smooth function in $\mathbb{R}$ satisfying 
\[ \supp(\psi) \subset [-1,1], \quad \int_{-1}^1 \psi(\xi) d\xi = 1.\]
Set $\varphi = \overline{\mathcal{F}}(\psi)$.
This is a rapidly decreasing function with $\varphi(0)=1$ and $|\varphi(x)|\leq 1$.
For $\varepsilon >0$ we set $\varphi_\varepsilon(x) = \varphi(\varepsilon x)$.
This satisfies $\mathcal{F}(\varphi_\varepsilon)(\xi)  = \psi(\xi/\varepsilon)/\varepsilon$.
The function $\varphi_\varepsilon$ can be extended to a holomorphic function in $\mathbb{C}$ satisfying 
$|\varphi_\varepsilon(x+y\sqrt{-1})|\leq  e^{2\pi\varepsilon |y|}$.
We have $\varphi_\varepsilon \to 1$ $(\varepsilon \to 0)$ uniformly over every compact subset of $\mathbb{C}$.
Set $f_\varepsilon(x) = \varphi_\varepsilon(x) f(x)$.
Note that $f_\varepsilon$ is a $L^2$ function.

Suppose $f$ satisfies the condition (1).
The Fourier transform $\mathcal{F}(f_\varepsilon) = \mathcal{F}(\varphi_\varepsilon) * \mathcal{F}(f)$ is 
supported in $[-a-\varepsilon, a+\varepsilon]$.
Since $f_\varepsilon \in L^2(\mathbb{R})$, we can apply to it the standard Paley--Wiener theorem (indeed this is a trivial part
of their theorem)
and conclude that 
$f_\varepsilon$ can be extended to a holomorphic function in $\mathbb{C}$ satisfying 
\[  |f_\varepsilon(x+y\sqrt{-1})| \leq e^{2\pi(a+\varepsilon)|y|} \norm{f_\varepsilon}_{L^\infty(\mathbb{R})} 
    \leq e^{2\pi(a+\varepsilon)|y|} \norm{f}_{L^\infty(\mathbb{R})} \quad 
    (\text{by Lemma \ref{lemma: exponential type} and $|\varphi_\varepsilon|\leq 1$}).\]
Then we extend $f$ to a meromorphic function in $\mathbb{C}$ by 
\[ f(z) = \varphi_\varepsilon(z)^{-1} f_\varepsilon(z)  \quad (\text{this is independent of $\varepsilon$ because of the unique continuation}). \]
Since $\varphi_\varepsilon \to 1$ uniformly over every compact subset of $\mathbb{C}$, we get 
\[ |f(x+y\sqrt{-1})| \leq e^{2\pi a|y|} \norm{f}_{L^\infty(\mathbb{R})}.\]
Thus $f$ becomes a holomorphic function satisfying (\ref{eq: exponential type}).

Next suppose $f$ satisfies (2).
Then the function $f_\varepsilon$ becomes a holomorphic function in $\mathbb{C}$ satisfying
$|f_\varepsilon(x+y\sqrt{-1})|\leq \mathrm{const} \cdot e^{2\pi(a+\varepsilon)|y|}$.
By the (difficult part of) Paley--Wiener theorem we get $\supp(\mathcal{F}(f_\varepsilon))\subset [-a-\varepsilon, a+\varepsilon]$.
The functions $f_\varepsilon$ converge to $f$ in the sense of tempered distributions as $\varepsilon \to 0$.
Thus $f$ satisfies $\supp(\mathcal{F}(f))\subset [-a,a]$.
\end{proof}

\begin{lemma} \label{lemma: compactness of B_1(V[a,b])}
Let $a<b$.
The space $B_1(V[a,b])$ introduced in Section \ref{subsection: embedding via signal processing}
is compact with respect to the distance $\boldsymbol{d}$ in (\ref{eq: distance on V[a,b]}).
\end{lemma}

\begin{proof}
First note that a sequence $\{f_n\}\subset B_1(V[a,b])$ converges to $f$ in the distance $\boldsymbol{d}$ if and only if 
 for any compact subset 
$K\subset \mathbb{R}$
\[ \lim_{n\to \infty} \norm{f_n-f}_{L^\infty(K)} = 0.\]
Set $c= 2\pi \max(|a|,|b|)$.
By Lemmas \ref{lemma: exponential type} and \ref{lemma: Paley--Wiener},
every $f\in B_1(V[a,b])$ can be extended to a holomorphic function in $\mathbb{C}$ satisfying 
\[ \forall x,y\in \mathbb{R}: \> |f(x+y\sqrt{-1})| \leq e^{c|y|}.\]
Then the compactness of $B_1(V[a,b])$ follows from the standard normal family argument
(Ahlfors \cite[Chapter 5, Section 5.4]{Ahlfors}):
If $\{f_n\}$ is a sequence of holomorphic functions in $\mathbb{C}$ uniformly bounded over every compact subset of $\mathbb{C}$, 
then a suitable subsequence converges to a holomorphic function uniformly over every compact subset.
\end{proof}

\begin{lemma}[Sampling theorem] \label{lemma: sampling theorem}
Let $a$ and $d$ be positive numbers with $2a d < 1$. 
Set $\Lambda = d \mathbb{Z}\subset \mathbb{R}$.
Then the following map is injective:
\[  V[-a,a] \to \ell^\infty(\Lambda), \quad f\mapsto f|_{\Lambda} =  (f(\lambda))_{\lambda\in \Lambda}.\]
Note that this statement is optimal because the function $\sin (2\pi x)$ belongs to $V[-1,1]$ and vanishes over $(1/2)\mathbb{Z}$.
\end{lemma}

\begin{proof}
Suppose there exists a nonzero $f\in V[-a,a]$ satisfying $f|_{\Lambda}=0$.
By the first main theorem of Nevanlinna (Hayman \cite[Section 1.3]{Hayman} and 
Noguchi--Winkelmann \cite[Section 1.1]{Noguchi--Winkelmann}), for $r>1$
\begin{equation} \label{eq: first main theorem}
 \int_1^r \#\{z\in \mathbb{C}|\, |z|\leq t, f(z)=0\} \frac{dt}{t} \leq \frac{1}{2\pi}\int_0^{2\pi} \log^+ |f(re^{\theta\sqrt{-1}})|d\theta
   + \mathrm{const} .
\end{equation}
Here $\log^+ x = \max(0, \log x)$.
From $f|_\Lambda =0$, the left-hand side is bounded from below by 
\[ \int_1^r \frac{2t}{d} \frac{dt}{t} + O(\log r) = \frac{2r}{d} + O(\log r).\]
From Lemma \ref{lemma: Paley--Wiener}, 
\[ \log^+ |f(re^{\sqrt{-1}\theta})| \leq  2\pi a r |\sin\theta|  + \mathrm{const}.\]
Hence the right-hand side of (\ref{eq: first main theorem}) is bounded by 
\[ a r\int_0^{2\pi} |\sin\theta|d\theta + \mathrm{const} = 4a r + \mathrm{const}.\]
Thus 
\[ \frac{2r}{d} \leq 4a r + O(\log r). \]
Letting $r\to \infty$, we get $2ad \geq 1$, which contradicts the assumption.
\end{proof}

\section{Technical main theorem and the proof of Theorem \ref{thm: embed minimal systems into continuous signals}}  
\label{section: embedding extensions of nontrivial minimal systems}

Here we formulate Theorem \ref{thm: technical main theorem}, which is technically the most important result of the paper.
Theorem \ref{thm: embed minimal systems into continuous signals} in Section \ref{subsection: embedding via signal processing}
follows from this.
The proof of Theorem \ref{thm: technical main theorem} occupies all the rest of the paper.
In Section \ref{subsection: ideas of the proof} we explains the ideas of the proof.

\subsection{Technical main theorem}  \label{subsection: technical main theorem}

Let $(Y,S)$ be a dynamical system.
$(Y,S)$ is said to have the \textbf{marker property}
if for any natural number $N$ there exists an open set $U\subset Y$ satisfying 
\[  U\cap S^n U= \emptyset \quad (\forall 1\leq n\leq N), \quad 
    Y = \bigcup_{n\in \mathbb{Z}} S^n U. \]
It is known that large classes of dynamical systems satisfy this condition.    
Nontrivial minimal systems obviously have the marker property.
If $(Y,S)$ is a finite dimensional system having no periodic points, then it satisfies the marker property; 
see \cite[Theorem 6.1]{Gutman 2}.
The authors don't know an example of dynamical systems which have no periodic points but don't have the marker property.

Let $(X,T)$ and $(Y,S)$ be dynamical systems with 
a continuous surjection $\Phi:X\to Y$ satisfying $\Phi\circ T = S\circ \Phi$.
We call $(X,T)$ an \textbf{extension} of $(Y,S)$, and $(Y,S)$ a \textbf{factor} of $(X,T)$.
Take two real numbers $a<b$.
We denote by $C_T(X, B_1(V[a,b]))$ the space of continuous maps $f:X\to B_1(V[a,b])$ satisfying 
$f\circ T = \sigma \circ f$.
Here the topology of $B_1(V[a,b])$ is given by the distance $\boldsymbol{d}$ in (\ref{eq: distance on V[a,b]}).
Note that $C_T(X,B_1(V[a,b]))$ is always nonempty because it contains the trivial map $f(x) =0$.
The space $C_T(X,B_1(V[a,b]))$ becomes a complete metric space with respect to the uniform distance 
\[ \sup_{x\in X} \boldsymbol{d}(f(x), g(x)), \quad (f, g\in C_T(X, B_1(V[a,b]))).\]

\begin{theorem}  \label{thm: technical main theorem}
Under the above settings, suppose that $\mdim(X,T)<b-a$ and
$(Y,S)$ has the marker property.
Then for a dense $G_\delta$ subset of $f\in C_T(X, B_1(V[a,b]))$ the map 
\[  (f,\Phi):  X\to B_1(V[a,b])\times Y, \quad x\mapsto (f(x), \Phi(x)) \]
is an embedding.
\end{theorem}

Every dense $G_\delta$ subset of $C_T(X, B_1(V[a,b]))$ is non-empty by the Baire category theorem.
So the theorem implies that $(X,T)$ can be embedded into the product system $B_1(V[a,b])\times Y$.

\subsection{Proof of Theorem \ref{thm: embed minimal systems into continuous signals}} 
\label{subsection: proof of continuous signal version of the main theorem}

Let $a<b$ be real numbers throughout this subsection.

\begin{lemma}  \label{lemma: existence of nontrivial map}
Let $(X,T)$ be a dynamical system with a non-periodic point $p\in X$. 
Then for a dense $G_\delta$ subset of $f\in C_T(X,B_1(V[a,b]))$ 
the function $f(p)$ is not shift-periodic, i.e. for any nonzero integer $n$ there exists a real number $t$ 
satisfying $f(p)(t+n) \neq f(p)(t)$.
\end{lemma}
\begin{proof}
We show that the following set is open and dense for any natural number $n$:
\[ \{ f\in C_T(X,B_1(V[a,b]))|\, f(p) \neq f(T^n p)\}.\]
This is obviously open. So it is enough to show its density.
Take any $f\in C_T(X, B_1(V[a,b]))$ and suppose $f(p) = f(T^n p)$.

Let $\psi(\xi)$ be a nonnegative smooth function in $\mathbb{R}$ satisfying 
\[ \supp (\psi) \subset \left[ -\frac{b-a}{2}, \frac{b-a}{2}\right], \quad \int_{-\infty}^\infty \psi(\xi) d\xi = 1.\]
Define $\varphi\in V[a,b]$ by 
\[ \varphi(t) = \exp\left(2\pi\sqrt{-1}\left(\frac{a+b}{2}\right) t \right) \overline{\mathcal{F}}(\psi)(t).\]
This satisfies $|\varphi(t)|\leq 1$ and $\varphi(0) = 1$.
The function $\varphi$ is rapidly decreasing. In particular we can find $K>0$ such that $\varphi(t) \leq K/(1+t^2)$.
Take a sufficiently large natural number $N=N(K)$.
Since $p\in X$ is not periodic, there exists an open neighborhood $U\subset X$ of $p$ satisfying 
$U\cap T^k U = \emptyset$ for $1\leq k\leq 2nN$.
Let $\alpha:X\to [0,1]$ be a continuous function satisfying 
$\supp(\alpha) \subset U$ and $\alpha(p)=1$.
For $x\in X$ we define $g(x)\in V[a,b]$ by 
\[ g(x)(t) = \sum_{k\in \mathbb{Z}} \alpha(T^k x) \varphi(t-k).\]
This is equivariant: $g(Tx)(t) = g(x)(t+1)$.
Since $N\gg 1$ we can assume $|g(x)(t)|\leq 2$, $g(p)(0) > 1/2$ and $g(p)(nN) < 1/2$.

Let $\varepsilon$ be a small positive number. We define $h\in C_T(X, B_1(V[a,b]))$ by 
\[ h(x) = \left(1-\frac{\varepsilon}{2}\right)f(x) + \frac{\varepsilon}{4} g(x).\]
This satisfies $|h(x)-f(x)|\leq \varepsilon$.  
\[ h(p)(nN) =  \left(1-\frac{\varepsilon}{2}\right)f(p)(nN) + \frac{\varepsilon}{4} g(p)(nN) 
   =   \left(1-\frac{\varepsilon}{2}\right)f(p)(0) + \frac{\varepsilon}{4} g(p)(nN) .\]
Since $g(p)(0) > 1/2$ and $g(p)(nN) < 1/2$, this is smaller than 
\[  h(p)(0) =  \left(1-\frac{\varepsilon}{2}\right)f(p)(0) + \frac{\varepsilon}{4} g(p)(0) .\]
Thus $h(T^{nN}p) \neq h(p)$. In particular $h(T^n p) \neq h(p)$.
\end{proof}

Theorem \ref{thm: embed minimal systems into continuous signals} is a special case of the next corollary.

\begin{corollary}  \label{cor: embed extensions of minimal systems into continuous signals}
If $(X,T)$ is an extension of a nontrivial minimal system with 
\[   \mdim(X,T)<b-a, \]
then we can embed it into the shift on $B_1(V[a,b])$.
\end{corollary}

\begin{proof}
$(X,T)$ has a nontrivial minimal factor $(Y,S)$.
Take $a<c_1<c_2<b$ with $\mdim(X,T) < c_1-a$.
From Lemma \ref{lemma: existence of nontrivial map}
there exists an equivariant continuous map $g:Y\to B_1(V[c_2,b])$ such that $Z=g(Y)$ is a nontrivial minimal system 
with respect to the shift.
Note that $Z$ also becomes a factor of $X$ and that it has the marker property.
Applying Theorem \ref{thm: technical main theorem}
to the factor map $X\to Z$, we can find an embedding of $(X,T)$ into the system $B_1(V[a,c_1])\times Z$, 
which becomes a subsystem of $B_1(V[a,b])$ by the embedding 
\[ B_1(V[a,c_1])\times B_1(V[c_2,b]) \to B_1(V[a,b]), \quad 
   (\varphi_1, \varphi_2) \mapsto \frac{1}{2}(\varphi_1 + \varphi_2).\]
\end{proof}

\begin{remark}    \label{remark: why we use continuous signals}
Although the above proof of Corollary \ref{cor: embed extensions of minimal systems into continuous signals}
is very simple, it exhibits the reason why continuous signals are more flexible than discrete ones.
The main trick above is to \textit{take two disjoint bands $[a,c_1]$ and $[c_2,b]$}.
This is possible because $a$ and $b$ are continuous parameters.
But the Hilbert cubes $([0,1]^N)^{\mathbb{Z}}$ have only the discrete parameter $N$.   
So we cannot apply the same trick.
\end{remark}

By the same argument as in the proof of (Theorem \ref{thm: embed minimal systems into continuous signals} $\Longrightarrow$
Theorem \ref{thm: main theorem}), we can deduce the next corollary from 
Corollary \ref{cor: embed extensions of minimal systems into continuous signals}.

\begin{corollary}  \label{cor: embed extensions of nontrivial minimal system into Hilbert cubes}
Let $N$ be a natural number, and $(X,T)$ an extension of a nontrivial minimal system with $\mdim(X,T)<N/2$.
Then we can embed $(X,T)$ into the shift on $\left([0,1]^N\right)^{\mathbb{Z}}$.
\end{corollary}

\subsection{Ideas of the proof}  \label{subsection: ideas of the proof}

Here we explain our approach to Theorem \ref{thm: technical main theorem}.
The proof is technically involved.
So we explain the proof of the \textit{toy-model} and discuss what are the main difficulties in the case of Theorem 
\ref{thm: technical main theorem}.

Suppose $(Y,S)$ is a \textit{zero dimensional} dynamical system with the marker property.
Here the zero dimensionality means that clopen sets (closed and open sets) form an open basis of the topology.
Let $\Phi:(X,T)\to (Y,S)$ be an extension.
We denote by $C_T\left(X, [0,1]^{\mathbb{Z}}\right)$ the space of equivariant continuous maps 
$f:X\to [0,1]^{\mathbb{Z}}$.
The next theorem is proved in \cite[Theorem 1.5]{Gutman--Tsukamoto}.

\begin{theorem} \label{thm: toy-model}
If $\mdim(X,T)<1/2$ then for a dense $G_\delta$ subset of $f\in C_T\left(X,[0,1]^\mathbb{Z}\right)$
the map 
\[  (f,\Phi): X\to [0,1]^\mathbb{Z}\times Y, \quad x\mapsto (f(x), \Phi(x)) \]
is an embedding.
\end{theorem}

We briefly explain the proof of this theorem, which is a prototype of the proof of Theorem \ref{thm: technical main theorem}.
Let $d$ be a distance on $X$.
Note the obvious equivalence:
\[  \text{embedding}  \Longleftrightarrow  \text{$\varepsilon$-embedding for all $\varepsilon>0$}. \]
Therefore it is enough to prove that the following set is dense and $G_\delta$:
\[ \bigcap_{n=1}^\infty \left\{f\in C_T\left(X,[0,1]^\mathbb{Z}\right) |\, 
   \text{$(f,\Phi)$ is a $(1/n)$-embedding w.r.t. $d$}\right\}. \]
This is obviously $G_\delta$.
So our main task is to prove the next proposition.

\begin{proposition}  \label{prop: toy-model}
For any $\delta>0$ and $f\in C\left(X,[0,1]^\mathbb{Z}\right)$ there exists 
$g\in C\left(X,[0,1]^\mathbb{Z}\right)$ satisfying the following.

\noindent 
(1) For all $x\in X$ and $t\in \mathbb{Z}$, $|f(x)(t)-g(x)(t)| < \delta$.

\noindent 
(2) $(g,\Phi):X\to [0,1]^\mathbb{Z}\times Y$ is a $\delta$-embedding with respect to $d$.
\end{proposition}

\begin{proof}
Take $0<\varepsilon <\delta$ such that 
\begin{equation*} 
   d(x,y)< \varepsilon \Longrightarrow  |f(x)(0) - f(y)(0)| < \delta.
\end{equation*}
From $\mdim(X,T)<1/2$, we can find $N>0$ such that 
\[  \widim_\varepsilon (X,d_n) < \frac{n}{2} \quad (\forall n\geq N). \]
Then for every $n\geq N$ we can construct an $\varepsilon$-embedding with respect to $d_n$ 
\[ G_n :X\to [0,1]^n = [0,1]^{\{0,1,2,\dots,n-1\}}  \]
satisfying $|G_n(x)(t)- f(x)(t)|< \delta$ for all $x\in X$ and $t=0,1,2,\dots,n-1$.
For the reason why such $G_n$ exists, see Lemma \ref{lemma: approximation by linear map} and Corollary \ref{cor: embeddings are dense}.

From the assumption on $Y$, there exists a clopen set $U\subset Y$ satisfying 
\[  U\cap S^n U = \emptyset \quad (\forall 1\leq n\leq N), \quad 
    Y = \bigcup_{n\in \mathbb{Z}} S^n U. \]
Take $x\in X$.  Let $E(x)$ be the set of integers $n$ satisfying $\Phi(T^n x) = S^n \Phi(x) \in U$.
For each $n\in E(x)$ we define the interval $I(x,n)\subset \mathbb{Z}$ by 
\[ I(x,n) = \left\{k\in \mathbb{Z}|\, \forall m\in E(x): |k-n|\leq |k-m|\right\}. \]
This is a kind of ``Voronoi diagram construction''.
(Voronoi diagram was first used by \cite{Gutman 1} in the context of mean dimension.)
From the assumption on $U$, the interval $I(x,n)$ is always finite and $\#I(x,n)>N$. 
We denote by $\alpha_{x,n}$ and $\beta_{x,n}$ the left and right end-points of $I(x,n)$ respectively.

We define $g(x)\in [0,1]^\mathbb{Z}$ by 
\[  g(x)(t) = G_{\# I(x,n)-1}\left(T^{\alpha_{x,n}}x \right)(t-\alpha_{x,n}), \]
where $n$ is the integer in $E(x)$ satisfying $\alpha_{x,n}\leq t < \beta_{x,n}$.
Roughly speaking, we attached the ``perturbation map'' $G_{\#I(x,n)-1}$ to each interval $I(x,n)$.
It is direct to check that $g:X\to [0,1]^\mathbb{Z}$ is continuous and equivariant.
We have 
\[  |g(x)(t) - f(x)(t)| = 
    \left| G_{\# I(x,n)-1}\left(T^{\alpha_{x,n}}x \right)(t-\alpha_{x,n}) - f\left(T^{\alpha_{x,n}} x\right)(t-\alpha_{x,n})\right|
     < \delta.\]

We prove that $(g,\Phi):X\to [0,1]^\mathbb{Z}\times Y$ is a $\delta$-embedding with respect to $d$.
Suppose $(g(x),\Phi(x)) = (g(x'),\Phi(x'))$ for some $x,x'\in X$.
The equation $\Phi(x)=\Phi(x')$ means that $E(x)=E(x')$ and $I(x,n)=I(x',n)$ for all $n\in E(x)$.
Take $n\in E(x)$ with $\alpha_{x,n} \leq 0 <  \beta_{x,n}$.
From $g(x)=g(x')$ 
\[ G_{\#I(x,n)-1}\left(T^{\alpha_{x,n}}x \right)  = G_{\#I(x,n)-1}\left(T^{\alpha_{x,n}}x' \right).  \]
Since the map $G_{\#I(x,n)-1}$ is an $\varepsilon$-embedding with respect to $d_{\#I(x,n)-1}$, we get 
\[  d(x,x') \leq d_{\#I(x,n)-1}\left(T^{\alpha_{x,n}} x, T^{\alpha_{x,n}} x' \right) < \varepsilon < \delta.\]
Here we have used $\alpha_{x,n}\leq 0 < \beta_{x,n}$ in the first inequality.
\end{proof}

The above proof is simple. But if we try a similar approach to Theorem \ref{thm: technical main theorem}, then 
we encounter the following four difficulties.

\begin{itemize}
   \item \textbf{Difficulty 1}. Theorem \ref{thm: toy-model}
    deals with discrete signals. But Theorem \ref{thm: technical main theorem}
   deals with continuous ones.
   So we need to convert the above procedure to the continuous setting.
   This is a rather straightforward issue.
   The main ingredient is a certain interpolation function prepared in Section \ref{section: interpolation}.

   \item \textbf{Difficulty 2}.  A crucial fact in the above proof of Theorem \ref{thm: toy-model} is that the set $U$ is clopen, 
   which implies that the map $g$ continuously depends on $x\in X$.   We cannot hope this in Theorem \ref{thm: technical main theorem}. 
   We overcome this difficulty by \textit{going one dimension higher}.
   We consider a certain Voronoi diagram in the \textit{plane} $\mathbb{R}^2$ and construct a tiling of the \textit{line} 
   \begin{equation}  \label{eq: construction of I(x,n), idea of the proof}
       \mathbb{R} = \bigcup_{n\in \mathbb{Z}} I(x,n) 
   \end{equation}    
   from the Voronoi diagram.
   This is a tricky idea first introduced by Lindenstrauss and the authors \cite{Gutman--Lindenstrauss--Tsukamoto}.
   We explain it in Section \ref{section: marker property and Voronoi tiling}.

   \item \textbf{Difficulty 3}.  All the intervals $I(x,n)$ are sufficiently long ($\#I(x,n) >N$) in the proof of Theorem \ref{thm: toy-model}.
   But some intervals $I(x,n)$ in (\ref{eq: construction of I(x,n), idea of the proof}) 
   may be short.
   We cannot construct a good perturbation over such short intervals.
   This is the most crucial difficulty.
   The key observation is that most part of the line is cover by sufficiently long intervals in
   (\ref{eq: construction of I(x,n), idea of the proof}).
   Then \textit{we take tax from these long intervals and use it for helping short intervals}.
   We explain this heuristic idea more precisely in Section \ref{section: marker property and Voronoi tiling}.

   \item \textbf{Difficulty 4}.  In the proof of Theorem \ref{thm: toy-model}, $I(x,n)$ are subsets of $\mathbb{Z}$.
   In particular they are locally constant with respect to $x\in X$.
   So it was easy to attach the perturbation maps $G_{\#I(x,n)-1}$ to $I(x,n)$.
   But the intervals $I(x,n)$ in (\ref{eq: construction of I(x,n), idea of the proof}) may continuously vary.
   Then it is more difficult to attach appropriate perturbation maps to $I(x,n)$ even if they are sufficiently long.
   We have to construct \textit{adjustable perturbation maps which can fit intervals of various length}. 
   This is a quite technical issue.  The construction is given in Section \ref{subsection: successive perturbations}, 
   which is based on Section \ref{section: linear maps from simplicial complex}.
\end{itemize}

After resolving all these difficulties, we finish the proof of Theorem \ref{thm: technical main theorem}
in Section \ref{subsection: construction of the map g}.

\section{Embeddings of simplicial complexes}   \label{section: linear maps from simplicial complex}

Here we prepare the method of constructing good maps from simplicial complex.
Every simplicial complex is assumed to be finite, i.e., it has only finitely many simplices.
The main result of this section is Lemma \ref{lemma: parametric genericity of embedding}.

In this section we use some standard ideas of real algebraic geometry.
A reference is Bochnak--Coste--Roy \cite{Bochnak--Coste--Roy}.
We will repeatedly use the Tarski--Seidenberg principle \cite[Proposition 2.2.7]{Bochnak--Coste--Roy}:
\textit{The image of a semi-algebraic set under a semi-algebraic map is also semi-algebraic}. 
The dimension of semi-algebraic sets is the algebraic dimension \cite[Definition 2.8.1]{Bochnak--Coste--Roy},
which does not behave pathologically.
(Indeed algebraic dimension coincides with topological dimension \cite[Theorem 2.3.6, Corollary 2.8.9]{Bochnak--Coste--Roy}.
But we don't need this fact.)

Let $P$ be a simplicial complex, and $V$ a real vector space.
A map $f:P \to V$ is said to be \textbf{simplicial} 
if for every simplex $\Delta\subset P$ it has the form 
\[  f\left(\sum_{k=0}^n \lambda_k v_k \right) = \sum_{k=0}^n \lambda_k f(v_k),  \quad
    \left(\lambda_k\geq 0, \> \sum_{k=0}^n\lambda_k =1\right),   \]
on $\Delta$, where $v_0, \dots, v_n$ are the vertices of $\Delta$.
The next lemma establishes the technique to approximate arbitrary continuous maps by simplicial ones.

\begin{lemma} \label{lemma: approximation by linear map}
Suppose $V$ is endowed with a norm $\norm{\cdot}$.
Let $(X,d)$ be a compact metric space with a continuous map $f:X\to V$.
Let $\varepsilon$ and $\delta$ be positive numbers satisfying 
\[ d(x,y)<\varepsilon \Longrightarrow \norm{f(x)-f(y)} < \delta.\]
Let $P$ be a simplicial complex, and $\pi:X\to P$ an $\varepsilon$-embedding with respect to $d$.
Then, after replacing $P$ by a sufficiently finer subdivision, there exists a 
simplicial map $g: P \to V$ satisfying 
\[ \norm{f(x)-g(\pi(x))} < \delta, \quad \forall x\in X.\]
\end{lemma}

\begin{proof}
This is proved in \cite[Lemma 2.1]{Gutman--Lindenstrauss--Tsukamoto}.
But we reproduce it here for the completeness.

For a vertex $v$ of $P$ we denote by $O(v)$ the open star around it (the union of the relative interiors of simplices containing $v$).
We can subdivide $P$ sufficiently finer so that 
$\diam \,\pi^{-1}(O(v)) < \varepsilon$ for all vertices $v$ of $P$.
Take any vertex $v\in P$.
If $\pi^{-1}(O(v))\neq \emptyset$ then we choose a point $x_v\in \pi^{-1}(O(v))$ and set $g(v) = f(x_v)$.
If $\pi^{-1}(O(v))= \emptyset$, then we choose $g(v)\in V$ arbitrarily. 
We define $g:P\to V$ by extending it linearly on every simplex.

Take $x\in X$.
Let $v_0, \dots, v_n$ be the vertices of $P$ satisfying $\pi(x)\in O(v_k)$.
We have $d(x,x_{v_k})<\varepsilon$ and hence
$\norm{f(x)-f(x_{v_k})} < \delta$. 
The point $g(\pi(x))$ is a convex combination of $f(x_{v_0}),\dots,f(x_{v_n})$.
Thus $\norm{f(x)-g(\pi(x))} < \delta$.
\end{proof}

Let $D$ be a natural number, and $P$ a simplicial complex of dimension $n$.
We denote by $V(P)$ the set of vertices of $P$. 
We naturally consider $P\subset \mathbb{R}^{V(P)}$.
The space of simplicial maps from $P$ to $\mathbb{R}^D$ is identified with 
the space $\mathrm{Hom}(\mathbb{R}^{V(P)}, \mathbb{R}^D)$ of linear maps from the vector space 
$\mathbb{R}^{V(P)}$ to $\mathbb{R}^D$.
This is endowed with the structure of a real algebraic manifold.
Its topology is the standard Euclidean topology
(not the Zariski topology).

\begin{lemma}   \label{lemma: space of non-injective linear maps}
Let $X$ be the space of simplicial maps $f:P\to \mathbb{R}^D$ which are not embeddings.
This is a semi-algebraic set in $\mathrm{Hom}(\mathbb{R}^{V(P)}, \mathbb{R}^D)$, and its codimension is 
greater than or equal to $D-2n$.
\end{lemma}
\begin{proof}
Consider 
\[ \{(x,y, f)\in P\times P \times  \mathrm{Hom}(\mathbb{R}^{V(P)}, \mathbb{R}^D)|\, x\neq y, f(x)=f(y)\}.\]
This is a semi-algebraic set in $\mathbb{R}^{V(P)}\times \mathbb{R}^{V(P)}\times  \mathrm{Hom}(\mathbb{R}^{V(P)}, \mathbb{R}^D)$.
Its projection to the factor  $\mathrm{Hom}(\mathbb{R}^{V(P)}, \mathbb{R}^D)$ is equal to $X$.
Thus $X$ is semi-algebraic by the Tarski--Seidenberg principle.
Let $A\subset V(P)$. 
We define $X_A$ as the space of simplicial maps 
$f:P\to \mathbb{R}^D$ such that $f(v)$ $(v\in A)$ are affinely dependent.
Its codimension is greater than or equal to $D-\#A+2$ by Sublemma \ref{sublemma: codimension of non-injective linear maps} below.
The space $X$ is contained in 
\[ \bigcup_{A\subset V(P), \#A \leq 2n+2} X_A.\]
Hence its codimension is greater than or equal to $D-2n$.

\begin{sublemma} \label{sublemma: codimension of non-injective linear maps}
Let $N$ be a natural number.
We consider the space $Y$ of non-injective linear maps $F:\mathbb{R}^N\to \mathbb{R}^D$.
Then its codimension in $\mathrm{Hom}(\mathbb{R}^N, \mathbb{R}^D)$ is greater than or equal to $D-N+1$.
\end{sublemma}
\begin{proof}
Set 
\[ Z = \{(F,x)\in \mathrm{Hom}(\mathbb{R}^N ,\mathbb{R}^D)\times \mathbb{P}^{N-1}(\mathbb{R})|\, Fx = 0\}.\]
Here $\mathbb{P}^{N-1}(\mathbb{R})$ is the $N-1$ dimensional projective space.
Let $\pi_1$ and $\pi_2$ be the projections from $Z$ to $ \mathrm{Hom}(\mathbb{R}^N ,\mathbb{R}^D)$ and $\mathbb{P}^{N-1}(\mathbb{R})$
respectively. For each $x\in \mathbb{P}^{N-1}(\mathbb{R})$ the space $\pi_2^{-1}(x)$ has dimension $D(N-1)$.
Thus $\dim Z \leq N-1 + D(N-1) = (D+1)(N-1)$.
Then the codimension of $Y=\pi_1(Z)$ is greater than or equal to $ND-(D+1)(N-1) = D-N+1$.
\end{proof}
\end{proof}

\begin{corollary} \label{cor: embeddings are dense}
If $D\geq 2n+1$ then embeddings $f:P\to  \mathbb{R}^D$ are dense in $\mathrm{Hom}(\mathbb{R}^{V(P)}, \mathbb{R}^D)$.
\end{corollary}

\begin{lemma} \label{lemma: parametric genericity of embedding}
Suppose $D\geq 2n+2$, and let $g:P\to \mathbb{R}^D$ be a simplicial map.
Then for an open dense subset of simplicial maps $f\in \mathrm{Hom}(\mathbb{R}^{V(P)}, \mathbb{R}^D)$,
the maps
\[ (1-t)f+tg: P \to \mathbb{R}^D, \quad x\mapsto (1-t)f(x)+tg(x) \]
become embeddings for all $0\leq t<1$.
\end{lemma}
\begin{proof}
Let $Z$ be the space of simplicial maps $f:P\to \mathbb{R}^D$ such that $f+g:P\to \mathbb{R}^D$ is not an embedding.
By Lemma \ref{lemma: space of non-injective linear maps}, 
this is semi-algebraic and its codimenion in $\mathrm{Hom}(\mathbb{R}^{V(P)}, \mathbb{R}^D)$ is greater than or equal to $D-2n\geq 2$.
Consider 
\[ \bigcup_{a\geq 0} aZ = \bigcup_{a\geq 0} \{af|\, f\in Z\} \subset \mathrm{Hom}(\mathbb{R}^{V(P)}, \mathbb{R}^D). \]
This is the image of the semi-algebraic map 
\[  [0,\infty) \times Z \to \mathrm{Hom}(\mathbb{R}^{V(P)},\mathbb{R}^D), \quad (a,f)\mapsto af. \]
So it is semi-algebraic by the Tarski--Seidenberg principle. 
Its codimension 
is greater than or equal to $D-2n-1\geq 1$; see \cite[Theorem 2.8.8]{Bochnak--Coste--Roy}.
Here we used real algebraic geometry essentially. 
We cannot hope a reasonable behavior of the topological dimension of $\bigcup_{a\geq 0} aZ$ if $Z$ is a fractal.

Then the codimension of the union
\begin{equation}  \label{eq: non parametric embedding set}
     \{f:P\to \mathbb{R}^D|\, \text{simplicial non-embedding}\}\cup  \bigcup_{a\geq 0} aZ.
\end{equation}
is also greater than or equal to 1.
In particular this is nowhere dense because the dimension of semi-algebraic sets does not increase under the operation of closure
\cite[Proposition 2.8.2]{Bochnak--Coste--Roy}. 
Any simplicial map $f:P\to \mathbb{R}^D$ in the complement of (\ref{eq: non parametric embedding set}) satisfies the required property.
\end{proof}

\section{Interpolation}  \label{section: interpolation}

In this section we prepare the technique of \textit{interpolations}.
This is used for converting discrete signals into continuous ones.
Every idea here is due to Beurling \cite[pp. 351-365]{Beurling}.
We follow his argument.
The construction in this section is somewhat ad hoc, and a more sophisticated approach is possible.
But we prefer the ad hoc approach because it is more elementary.

Let $l$ and $\rho$ be positive numbers with $l\rho \in \mathbb{Z}$.

\begin{notation}
For two quantities $x$ and $y$ we write 
\[ x\lesssim y \]
if there exists a positive constant $C$ depending only on $l$ and $\rho$ satisfying 
\[ x\leq C y.\]
\end{notation}

Let $\Lambda\subset \mathbb{R}$ be a \textbf{multiset}.
`'Multi'' means that some points may have multiplicity.
For integers $n$ we set $\Lambda_n = \Lambda\cap [nl, (n+1)l)$.
The notation $\Lambda_n$ is used only in this section.

\begin{condition} \label{condition: interpolation}
(1) $\inf_{0\neq \lambda\in \Lambda} |\lambda| \geq 1/\rho$.

\noindent 
(2) For all integers $n$, we have $\# \Lambda_n \leq l\rho$. Here $\#(\cdot)$ is the counting with multiplicity.
For example, $\#\{1,1,1,2,3\} = 5$.

\noindent 
(3) For all nonzero integers $n$, we have $\#\Lambda_n = l\rho$.
\end{condition}

\begin{lemma}  \label{lemma: interpolation}
Suppose $\Lambda$ satisfies Condition \ref{condition: interpolation} (1), (2), (3).
Then 
\[  f(z) = \lim_{A\to \infty} \prod_{\lambda\in \Lambda, 0<|\lambda|<A} \left(1-\frac{z}{\lambda}\right) \]
defines a holomorphic function in $\mathbb{C}$ satisfying $f(0)=1$ and $f(\lambda) = 0$ for all nonzero $\lambda\in \Lambda$.
Moreover for all $z\in \mathbb{C}$
\begin{equation}  \label{eq: growth of interpolation}
   |f(z)| \lesssim (1+|z|)^{5l\rho} e^{\pi\rho |y|}, \quad (\text{$y$ is the imaginary part of $z$}) .
\end{equation}
The above product takes the multiplicity into account.
For example if a nonzero $\lambda$ appears twice in $\Lambda$ then the factor $(1-z/\lambda)$ appears twice in the product.
\end{lemma}

\begin{proof}
We should keep in mind the following simple fact, which is a toy-model of the statement.
The function 
\[ \frac{\sin z}{z} = \lim_{A\to \infty} \prod_{0<|n|<A} \left(1-\frac{z}{n\pi}\right) \]
is holomorphic, and its growth is $O(e^{|y|})$.

We use the notation $\sum^*$ and $\prod^*$ for indicating sum and product over $\lambda\neq 0$.
For example we write 
\[  f(z) = \lim_{A\to \infty} \prod_{\lambda\in \Lambda, |\lambda|<A}^* \left(1-\frac{z}{\lambda}\right).\]
First we need to show the convergence of $f(z)$.
A slightly delicate point is that this is a \textit{conditional} convergence, i.e. the sum $\sum_{\lambda\in \Lambda}^* 1/|\lambda|$
diverges.

For $|z/\lambda|< 1$
\[ \log\left(1-\frac{z}{\lambda}\right) = -\frac{z}{\lambda} - \frac{z^2}{2\lambda^2} - \frac{z^3}{3\lambda^3} - \dots,\]
From Condition \ref{condition: interpolation} (2), we can easily prove
that for any $B>0$ 
\begin{equation}  \label{eq: estimate of higher terms}
     \sum_{\lambda \in \Lambda, |\lambda|>B} \frac{1}{|\lambda|^k} \lesssim \frac{1}{B^{k-1}}  \quad (k\geq 2). 
\end{equation}    
From Condition \ref{condition: interpolation} (3), for any $n\geq 2$ 
\begin{equation} \label{eq: cancellation of leading term}
    \left| \sum_{\lambda\in \Lambda_n} \frac{1}{\lambda} + \sum_{\lambda\in \Lambda_{-n}} \frac{1}{\lambda}\right| 
    \lesssim \frac{1}{n^2}.
\end{equation}
This implies the convergence of 
\[ \lim_{A\to \infty} \sum_{\lambda\in \Lambda, |\lambda|<A}^* \frac{1}{\lambda}.\]
Therefore $f(z)$ becomes a holomorphic function satisfying $f(0) =1$ and $f(\lambda) =0$ for all nonzero $\lambda\in \Lambda$.

Next we estimate the growth of $f$ on the real line. 
Suppose $x>0$ and let $k$ be the integer with $kl\leq x< (k+1)l$.
We assume $k>0$. The case $k=0$ is easier and can be discussed in a similar way.
\begin{itemize}
  \item For $\lambda\in \Lambda_n$ with $n\leq -2$ or $n\geq k+1$, we have $|1-x/\lambda|\leq 1-x/(n+1)l$. So 
\[  \prod_{\lambda\in \Lambda_n} \left|1-\frac{x}{\lambda}\right| \leq  \left|1-\frac{x}{(n+1)l}\right|^{l\rho}. \]
  \item For $\lambda\in \Lambda_n$ with $1\leq n < k$, we have $|1-x/\lambda| \leq  x/(nl) -1$. So 
\[  \prod_{\lambda\in \Lambda_n} \left|1-\frac{x}{\lambda}\right| \leq \left|1-\frac{x}{nl}\right|^{l\rho}.\]
\end{itemize}
The factors for $n=-1,0,k$ should be treated exceptionally.
The modulus $|f(x)|$ is bounded by 
\begin{equation*}
   \begin{split}
   & \prod^*_{\lambda\in \Lambda_{-1}\cup\Lambda_0\cup\Lambda_k} \left|1-\frac{x}{\lambda}\right|   \cdot
    \lim_{A\to \infty} \prod_{|n|<A, n\neq 0, k, (k+1)} \left|1-\frac{x}{nl}\right|^{l\rho} \\
   & =    \prod^*_{\lambda\in \Lambda_{-1}\cup\Lambda_0\cup\Lambda_k} \left|1-\frac{x}{\lambda}\right|  \cdot
         \left|\frac{\sin \frac{\pi x}{l}}{ \frac{\pi x}{l} \left(1-\frac{x}{kl}\right)\left(1-\frac{x}{(k+1)l}\right)}\right|^{l\rho}    .
   \end{split}
\end{equation*}   
The first factor is easy to estimate:
\[ \prod_{\lambda\in \Lambda_{-1}\cup\Lambda_0\cup\Lambda_{k}} \left|1-\frac{x}{\lambda}\right| \lesssim 
   (1+x)^{3l\rho}.\]
Set $t=x/l$. 
\[ \frac{\sin \frac{\pi x}{l}}{ \frac{\pi x}{l} \left(1-\frac{x}{kl}\right)\left(1-\frac{x}{(k+1)l}\right)}  = 
   \frac{k(k+1) \sin\pi t}{\pi t(k-t)(k+1-t)}.\]
From the mean value theorem,
\[ \left|\frac{\sin \pi t}{t}\right| \leq \pi, \quad 
    \left|\frac{\sin\pi t}{k-t}\right| \leq \pi, \quad 
    \left|\frac{\sin\pi t}{k+1-t}\right| \leq \pi.\]
Hence 
\[ \left| \frac{k(k+1) \sin\pi t}{\pi t(k-t)(k+1-t)}\right|  \lesssim k(k+1) \lesssim (1+x)^2.\]
Therefore 
\[ |f(x)| \lesssim (1+x)^{5l\rho}.\]
The case $x<0$ is the same and we get 
\[ |f(x)| \lesssim (1+|x|)^{5l\rho}.\]

Next we estimate $|f(y\sqrt{-1})|$. Suppose $y>0$.
For $r>0$ we set $n(r) = \#(\Lambda\cap (-r,r))$.
This is bounded by 
\[ n(r) \leq C + 2\rho r \]
where $C$ is a positive constant depending only on $l$ and $\rho$.
We have $|f(y\sqrt{-1})|^2 = \prod^*_{\lambda\in \Lambda} (1+y^2/\lambda^2)$.
Hence 
\[ 2 \log |f(y\sqrt{-1})|  = \sum^*_{\lambda\in \Lambda}\log\left(1+\frac{y^2}{\lambda^2}\right) 
   = \int_0^\infty  \log\left(1+\frac{y^2}{r^2}\right) d n(r) .\]
Using the integration by parts, this is equal to 
\[   2y^2 \int_0^\infty \frac{n(r)}{r(r^2+y^2)} dr .\]
From $n(r) \leq C + 2\rho r$, 
\[ \log |f(y\sqrt{-1})|  \leq C y^2 \int_0^\infty \frac{dr}{r(r^2+y^2)}  + 2\rho y^2 \int_0^\infty \frac{dr}{r^2+y^2}
     = C\int_0^\infty \frac{dr}{r(r^2+1)}  +  \pi \rho y. \]
Thus 
\[ |f(y\sqrt{-1})| \lesssim e^{\pi \rho y}.\]
The case $y<0$ is the same. So we get 
\[  |f(y\sqrt{-1})| \lesssim e^{\pi\rho |y|}.  \]

Finally we show that $|f(z)|$ grows at most exponentially.
Let $z=x+y\sqrt{-1}$. We consider the case $x, y>0$. Other cases are the same.
Let $k$ be the integer with $kl\leq x<(k+1)l$.
Set 
\[ \Lambda' = \Lambda \setminus (\Lambda_{k-1}\cup \Lambda_k\cup \Lambda_{k+1}).\]
We estimate 
\[  \prod^*_{\lambda\in \Lambda_{k-1}\cup \Lambda_k\cup \Lambda_{k+1}} \left|1-\frac{z}{\lambda}\right|
     \lesssim (1+|z|)^{3l\rho}.\]
\begin{equation*}
   \begin{split}
    \lim_{A\to \infty} \prod^*_{\lambda\in \Lambda', |\lambda|<A} \left|1-\frac{z}{\lambda}\right|^2 
   = \lim_{A\to \infty} \prod^*_{\lambda\in \Lambda', |\lambda|<A} \left\{\left(1-\frac{x}{\lambda}\right)^2 + \frac{y^2}{\lambda^2}\right\}  \\
   = \left\{\lim_{A\to \infty} \prod^*_{\lambda\in \Lambda', |\lambda|<A} \left(1-\frac{x}{\lambda}\right)^2\right\} \cdot 
      \prod^*_{\lambda\in \Lambda'} \left\{ 1+ \frac{y^2}{(\lambda-x)^2}\right\} .
    \end{split}
\end{equation*}      
As in the proof of $|f(x)|\lesssim (1+|x|)^{5l\rho}$ we estimate 
\[  \lim_{A\to \infty} \prod^*_{\lambda\in \Lambda', |\lambda|<A} \left(1-\frac{x}{\lambda}\right)^2  \lesssim (1+x)^{12l\rho}.\]
As in $|f(y\sqrt{-1})| \lesssim e^{\pi\rho |y|}$ 
\[  \prod^*_{\lambda\in \Lambda'} \left\{ 1+ \frac{y^2}{(\lambda-x)^2}\right\} \lesssim e^{2\pi\rho |y|}.\]
Therefore we conclude that $|f(z)|$ grows at most exponentially.

We have proved that $f(z)$ is of exponential type with $|f(x)|\lesssim (1+|x|)^{5l\rho}$ and 
$|f(y\sqrt{-1})| \lesssim e^{\pi\rho |y|}$.
Then we can prove (\ref{eq: growth of interpolation}) by the Phragm\'{e}n--Lindel\"{o}f principle.
For example, in the first quadrant ($x,y\geq 0$), we apply the Phragm\'{e}n--Lindel\"{o}f principle to the function 
\[  (1+z)^{-5l\rho} e^{\pi\rho\sqrt{-1} z} f(z) \]
and conclude (\ref{eq: growth of interpolation}).
\end{proof}

\begin{lemma}  \label{lemma: continuity of interpolation}
For any positive numbers $r$ and $\varepsilon$ there exists $B_1= B_1(r,\varepsilon, l,\rho)>0$ satisfying the following.
Suppose $\Lambda\subset \mathbb{R}$ satisfies Conditions \ref{condition: interpolation} (1), (2), (3). Then
\[  \left|1-\lim_{A\to \infty} \prod_{\lambda\in \Lambda, B_1<|\lambda|<A} \left(1-\frac{z}{\lambda}\right)\right| < \varepsilon \quad 
    (|z|\leq r).\]
\end{lemma}

\begin{proof}
For $|z/\lambda|<1$ 
\[ \log\left(1-\frac{z}{\lambda}\right) = -\frac{z}{\lambda} - \frac{z^2}{2\lambda^2}-\frac{z^3}{3\lambda^3} - \dots.\]
From (\ref{eq: estimate of higher terms}) and (\ref{eq: cancellation of leading term}), for $B_1>0$
\[\lim_{A\to \infty} \left|\sum_{\lambda\in \Lambda, B_1< |\lambda| < A} \frac{1}{\lambda} \right| \lesssim  \frac{1}{B_1}, \quad 
    \sum_{\lambda\in \Lambda, |\lambda|>B_1} \frac{1}{|\lambda|^k} \lesssim \frac{1}{B_1^{k-1}} \quad
     (k\geq 2).\]
Thus for sufficiently large $B_1$
\[ \left| \lim_{A\to \infty} \sum_{\lambda\in \Lambda, B_1<|\lambda|<A} \log\left(1-\frac{z}{\lambda}\right) \right| \lesssim 
   \frac{|z|  + |z|^2 }{B_1}.\]
\end{proof}

We need to relax the conditions on $\Lambda$.
Let $\Lambda\subset \mathbb{R}$ be a multiset satisfying 
Conditions \ref{condition: interpolation} (1) and (2) but not necessarily (3).
For each nonzero integer $n$ we add $nl$ to $\Lambda$ with multiplicity $(l\rho-\#\Lambda_n)$.
We denote by $\Lambda^+$ the resulting multiset and call it the \textbf{saturation} of $\Lambda$.
This satisfies $\Lambda\subset \Lambda^+$ and all the three conditions of Condition \ref{condition: interpolation}.
(The construction of $\Lambda^+$ is the most ad hoc part of the argument.)

Let $\tau$ be a positive number.
Let $\psi(\xi)$ be a nonnegative smooth function in $\mathbb{R}$ satisfying 
\[  \supp(\psi)\subset \left[-\frac{\tau}{2}, \frac{\tau}{2}\right], \quad \int_{-\infty}^\infty \psi(\xi)d\xi = 1.\]
Then the inverse Fourier transform $\overline{\mathcal{F}}(\psi)$ is a rapidly decreasing function satisfying 
\[ \overline{\mathcal{F}}(\psi)(0) = 1,
  \quad \left|\overline{\mathcal{F}}(\psi)(x+y\sqrt{-1})\right| \leq   e^{\pi \tau |y|}.\]
We define a function $\varphi_\Lambda$ by 
\begin{equation}   \label{eq: rapidly decreasing interpolation}
   \varphi_\Lambda(x) = \overline{\mathcal{F}}(\psi)(x)
   \left\{\lim_{A\to\infty}  \prod_{\lambda\in \Lambda^+, 0<|\lambda|<A}\left(1-\frac{x}{\lambda}\right)\right\}.
\end{equation}
From Lemmas \ref{lemma: Paley--Wiener} and \ref{lemma: interpolation}
\begin{itemize}
   \item $\varphi_\Lambda$ belongs to the Banach space $V[-(\rho+\tau)/2, (\rho+\tau)/2]$.
   \item $\varphi_\Lambda(0)=1$ and $\varphi_\Lambda(\lambda)=0$ for all nonzero $\lambda\in \Lambda$.
   \item $\varphi_\Lambda$ is a rapidly decreasing. In particular there exists $K>0$ depending only on $l,\rho, \tau$ such that 
   \begin{equation}  \label{eq: damping of rapidly decreasing interpolation}
       |\varphi_\Lambda(x)| \leq \frac{K}{1+|x|^2} .
   \end{equation}   
\end{itemize}
Note that $\varphi_\Lambda$ depends on $l$, $\rho$ and $\tau$ although they are not explicitly written in the notation.
In the proof of Theorem \ref{thm: technical main theorem} the numbers $l$, $\rho$ and $\tau$ are fixed in the beginning of the argument.
So this does not cause a confusion.

\begin{lemma} \label{lemma: continuity of rapidly decreasing interpolation}
For any positive numbers $r$ and $\varepsilon$ there exists $B_2 = B_2(r,\varepsilon, l,\rho,\tau)>0$ satisfying the following.
Suppose $\Lambda, \Lambda'\subset \mathbb{R}$ satisfy Conditions \ref{condition: interpolation} (1) and (2).
If $\Lambda\cap [-B_2,B_2] = \Lambda'\cap [-B_2,B_2]$ then 
\[ |\varphi_\Lambda(x)-\varphi_{\Lambda'}(x)| < \varepsilon    \quad (|x|\leq r).  \]
\end{lemma}

\begin{proof}
If $\Lambda\cap [-B_2,B_2] = \Lambda'\cap [-B_2,B_2]$ then 
the saturations satisfy $\Lambda^+\cap [-B_2+l, B_2-l]= (\Lambda')^+\cap [-B_2+l, B_2-l]$.
Thus for $B_2>l$
\[ \frac{\varphi_{\Lambda'}(x)}{\varphi_\Lambda(x)} = \lim_{A\to \infty} 
   \frac{\prod_{\lambda\in (\Lambda')^+, B_2-l < |\lambda|<A} \left(1-\frac{x}{\lambda}\right)}
   {\prod_{\lambda\in \Lambda^+, B_2-l<|\lambda|<A} \left(1-\frac{x}{\lambda}\right)}.\]
From Lemma \ref{lemma: continuity of interpolation}, for sufficiently large $B_2$
\[ \left|1-\frac{\varphi_{\Lambda'}(x)}{\varphi_{\Lambda}(x)} \right| < \frac{\varepsilon}{K}    \quad (|x|\leq r).  \]
Then by (\ref{eq: damping of rapidly decreasing interpolation}) 
\[  |\varphi_\Lambda(x)-\varphi_{\Lambda'}(x)| = |\varphi_\Lambda(x)|  \left|1-\frac{\varphi_{\Lambda'}(x)}{\varphi_{\Lambda}(x)} \right| 
   < \varepsilon  \quad (|x|\leq r).\]
\end{proof}

\section{Voronoi diagram and weight functions}  \label{section: marker property and Voronoi tiling}

Here we introduce a tiling of $\mathbb{R}$.
This will be the basis of our
perturbation procedure. 
The key ingredient of the construction is \textit{dynamical Voronoi diagram}.
This is first introduced by Lindenstrauss and the authors \cite{Gutman--Lindenstrauss--Tsukamoto}.
We will consider a Voronoi diagram in the plane and cut it by the real line.
This gives a tiling of $\mathbb{R}$, which has several nice properties revealed in this section.
We would like to remark that our use of Voronoi diagram is conceptually influenced by the works of 
Lightwood \cite{Lightwood 1, Lightwood 2}, which 
study the $\mathbb{Z}^2$-version of the Krieger embedding theorem in symbolic dynamics.

Throughout this section, $(Y,S)$ is a dynamical system with the marker property.
Let $C, L_0, L_1$ be positive numbers.
We fix a natural number $L$ satisfying 
\begin{equation} \label{eq: choice of L}
  L > 4L_1+1 + 4CL_0(4L_0+3).
\end{equation}
The marker property condition is used in the next lemma.
\begin{lemma} \label{lemma: marker property}
There exist an integer $M >L$ and a continuous function $h:Y\to [0,1]$ satisfying 
the following two conditions.

\noindent 
(1) $\supp (h) \cap S^n(\supp (h)) = \emptyset$ for all $1\leq n\leq L$.

\noindent 
(2) $Y=\bigcup_{n=0}^{M-1} S^n(h^{-1}(1))$.
\end{lemma}

\begin{proof}
By the definition of the marker property, there exists an open set $U\subset Y$ satisfying 
\[  U \cap S^n U = \emptyset \quad (1\leq n\leq L), \quad Y= \bigcup_{n\in \mathbb{Z}} S^n U.\]
We can find $M>L$ and a compact set $K\subset U$ satisfying 
\[  Y = \bigcup_{n=0}^{M-1} S^n K. \]
Take a continuous function $h:Y\to [0,1]$ satisfying 
$\supp(h)\subset U$ and $h=1$ on $K$.
Then this satisfies the required properties.
\end{proof}

Take $x\in Y$. We consider the Voronoi diagram with respect to the set  
\[ \left\{ \left(n, \frac{1}{h(S^n x)}\right)|\, n\in \mathbb{Z}, h(S^n x) \neq 0\right\} \subset \mathbb{R}^2.\]
For $n\in \mathbb{Z}$ with $h(S^n x) \neq 0$ we define 
\[ V(x,n) = \{u\in \mathbb{R}^2|\, \forall m\in \mathbb{Z}: |u-(n,1/h(S^n x))|\leq |u-(m, 1/h(S^m x))|\}.\]
If $h(S^n x) =0$ then we set $V(x,n) = \emptyset$.
These form a Voronoi partitioning of $\mathbb{R}^2$:
\[ \mathbb{R}^2 = \bigcup_{n\in \mathbb{Z}} V(x,n).\]
We consider $\mathbb{R} = \mathbb{R}\times \{0\}$ as a subset of $\mathbb{R}^2$ and set 
$I(x,n) = \mathbb{R}\cap V(x,n)$.
These intervals form a tiling of $\mathbb{R}$:
\[ \mathbb{R} = \bigcup_{n\in \mathbb{Z}} I(x,n).\]
See figure \ref{fig: Voronoi tiling}.
\begin{figure}
    \centering
    \includegraphics[,bb= 0 0 515 260]{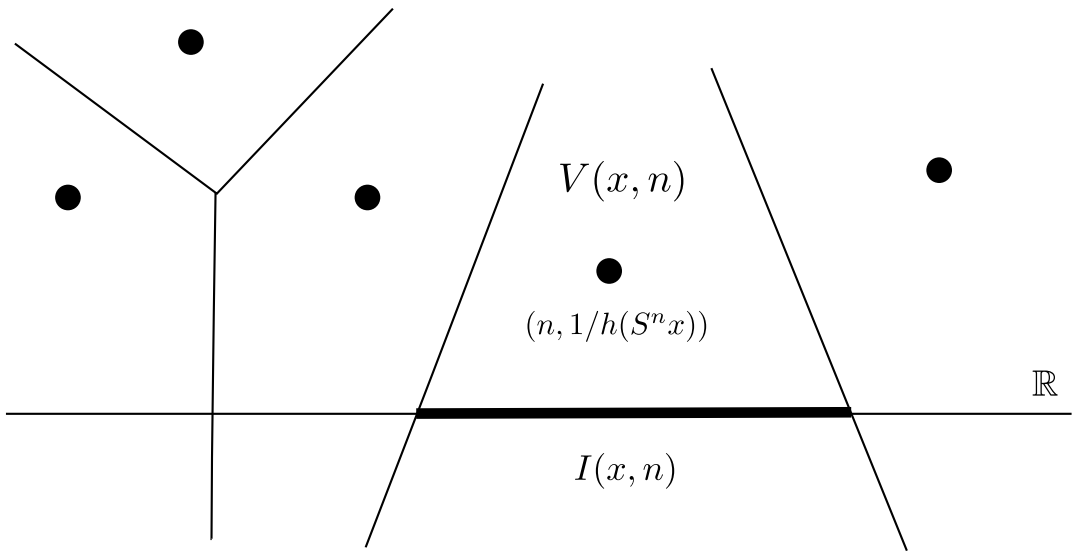}
    \caption{Voronoi diagram and $I(x,n)$.} \label{fig: Voronoi tiling}
\end{figure}
This construction is naturally \textit{dynamical}. Namely we have 
\[  I(Sx, n) = -1 + I(x,n+1). \]
The key point of the construction is that \textit{the interval $I(x,n)$ depends continuously on $x\in Y$}:
Suppose $I(x,n)$ is not a point. Then for any $\varepsilon>0$ the Hausdorff distance between $I(x,n)$ and $I(y,n)$ is less than 
$\varepsilon$ if $y\in Y$ is sufficiently close to $x$.

We define 
\[ \partial(x) = \bigcup_{n\in \mathbb{Z}} \partial I(x,n).\]
(Of course, $\partial [a,b] = \{a,b\}$.)
For $r>0$ we set 
$\partial_r [a,b] = [a-r,a+r] \cup [b-r,b+r]$.
($\partial_r \emptyset = \emptyset$.)
Set 
\[ \partial(x,r) = \bigcup_{n\in \mathbb{Z}} \partial_r I(x,n).\]

\begin{lemma} \label{lemma: properties of tiling}
(1) $I(x,n) \subset (n-M/2, n+M/2)$.

\noindent 
(2) For any $r>0$
\[ \limsup_{R\to \infty} \frac{1}{R} \sup_{a\in \mathbb{R}, x\in Y} \#\left([a,a+R]\cap \partial(x,r) \cap \mathbb{Z}\right) \leq \frac{4r+2}{L}.\]
Moreover (we denote by $|\cdot|$ the Lebesgue measure of $\mathbb{R}$)
\[ \limsup_{R\to \infty} \frac{1}{R}   \sup_{a\in \mathbb{R}, x\in Y} \left|[a,a+R]\cap \partial(x,r)\right| \leq \frac{4r}{L}. \]
\end{lemma}
\begin{proof}
(1) By Lemma \ref{lemma: marker property} (2), 
there exist integers $l\leq n\leq m$ with $h(S^l x) = h(S^m x) =1$ and $n-l, m-n < M$.
Let $t\in I(x,n)$.
If $t\geq n$ then
\[  |(t,0)-(n,1/h(S^n x))| \leq |(t,0)-(m,1)| \]
implies $t-n<M/2$. In the same way, if $t<n$ then we get $n-t<M/2$.

(2) By the above (1) 
\[ [a,a+R] \subset \bigcup_{-M/2<n-a<R+M/2} I(x,n).\]
Therefore $[a,a+R]\cap \partial(x,r)$ is contained in
\[ [a,a+r]\cup [a+R-r,a+R] \cup  \bigcup_{-M/2<n-a<R+M/2} \partial_r I(x,n). \]
The number of integers $n \in (a-M/2,a+R+M/2)$ satisfying $h(S^n x) \neq 0$ is bounded  by 
\[ 1 + \frac{R+M}{L} \]
because $\supp(h)\cap S^m(\supp(h)) = \emptyset$ for $1\leq m\leq L$.
Thus $ \#\left([a,a+R]\cap \partial(x,r) \cap \mathbb{Z}\right)$ is bounded by 
\[ 2(r+1)  + 2(2r+1)\left(1+\frac{R+M}{L}\right).\]
Dividing this by $R$ and letting $R\to \infty$, we get the result.
Another statement can be proved in the same way.
\end{proof}

Now we have come to the core of the argument.
In the proof of Theorem \ref{thm: technical main theorem}
we will have the following dichotomy:
Take $m\in \mathbb{Z}$. 
The point $m$ is said to be \textit{wild} if $m \in \partial(x,L_0 -4)$.
Otherwise it is \textit{tame}.
Here $L_0-4$ is just a technical number.
Readers may think that a point is wild if it is close to $\partial(x)$.
Tame points can be handled easily.
A main problem is how to deal with wild points.

The following is the idea behind the above dichotomy.
It is not difficult to control band-limited functions over sufficiently long intervals.
But it is impossible to control them over short intervals because of their band-limited nature.
(Intuitively speaking, band-limited functions cannot have very small fluctuation.)
As a consequence, if the length of $I(x,n)$ is sufficiently larger than $L_0$ (which will be chosen appropriately later),
then we can construct a good perturbation of band-limited functions over it.
But if it is less than $L_0$, then we cannot construct a perturbation there.
So the problem is how to deal with short intervals.

We overcome the difficulty by introducing \textit{tax system}.
Long intervals are good for our perturbation procedure.
But some of them are unnecessarily long.
We consider $I(x,n)$ \textit{too long} if $|I(x,n)|>L_1$.
Then we take \textit{tax} $|I(x,n)|-L_1$ from too long intervals, and use it for the care of wild points.
If \textit{every lattice point becomes happy}, then the proof is done.

The next lemma is the basis of our tax system.
Intuitively, it means that \textit{the sum of tax is larger than the cost of social security}.
For $x\in \mathbb{R}$ we set $x^+ = \max(x,0)$.

\begin{lemma} \label{lemma: there are enough long intervals}
There exists an integer $R>M$ such that for all $x\in Y$ and $a\in \mathbb{R}$
\begin{equation} \label{eq: there are enough long intervals}
   \sum_{a\leq n\leq a+R} (|I(x,n)|-L_1)^+ \geq C \sum_{a\leq n\leq a+R} (L_0-\dist(n,\partial(x)))^+, 
\end{equation}   
where $\dist(n,\partial(x)) = \min_{t\in \partial(x)}|n-t|$.
In the above two sums, $n$ runs over $\mathbb{Z}\cap [a,a+R]$.
Intuitively the left-hand side is the sum of tax in the region $a\leq n\leq a+R$ and the right-hand side is the cost of 
social security there.
\end{lemma}
\begin{proof}
By Lemma \ref{lemma: properties of tiling} (2), the right-hand side of (\ref{eq: there are enough long intervals}) is bounded by 
\[ CL_0 \#(\mathbb{Z}\cap [a,a+R]\cap \partial(x,L_0)) \leq CL_0R \frac{4L_0+3}{L} \]
for $R\gg 1$.
Let $A(x)$ be the set of integers $n$ with $|I(x,n)|\geq 2L_1$.
For $n\in A(x)$ we have $(|I(x,n)|-L_1)^+ \geq |I(x,n)|/2$.
So the left-hand side of (\ref{eq: there are enough long intervals}) is bounded from below by 
\[ \frac{1}{2} \sum_{n\in A(x)\cap [a,a+R]} |I(x,n)|. \]
By Lemma \ref{lemma: properties of tiling} (1) we have 
\[ [a+M/2, a+R-M/2]\subset \bigcup_{a\leq n\leq a+R} I(x,n).\]
Hence 
\[  [a+M/2, a+R-M/2]\setminus \partial(x,L_1) \subset \bigcup_{n\in A(x)\cap [a,a+R]} I(x,n), \]
\begin{equation*}
  \begin{split}
    \sum_{n\in A(x)\cap [a,a+R]} |I(x,n)|  &\geq R-M -|[a+M/2, a+R-M/2]\cap \partial(x,L_1)| \\
   & \geq  R-M - (R-M) \frac{4L_1+1}{L} \quad \text{for $R\gg1$ by Lemma \ref{lemma: properties of tiling} (2)} \\
   & \geq \frac{R}{2} \left(1-\frac{4L_1+1}{L}\right) \quad  (R\gg 1) \\
   & \geq 2CL_0R\frac{4L_0+3}{L} \quad \text{by (\ref{eq: choice of L})}.
  \end{split}
\end{equation*}
Combing the above estimates, we get (\ref{eq: there are enough long intervals}).
\end{proof}

The next lemma is our tax system.

\begin{lemma}  \label{lemma: construction of weight}
 There exists a continuous map (called \textbf{weight})
 \[ Y\to ([0,1]^{R+1})^{\mathbb{Z}}, \quad x\mapsto w(x) = (w_n)_{n\in \mathbb{Z}}, \quad 
     w_n = (w_{n0}, w_{n1}, \dots, w_{nR}) \in [0,1]^{R+1}, \]
     satisfying the following.
 
 \noindent 
 (1) The map is equivariant: $w_n(Sx) = w_{n+1}(x)$.

\noindent 
(2) If $|I(x,n)|\leq L_1$ then $w_n=(0,\dots, 0)$.

\noindent 
(3) For all $n\in \mathbb{Z}$ 
\[ \#\{m|\, w_{nm}>0\} \leq 1 + \frac{1}{C}  (|I(x,n)|-L_1)^+.\]

\noindent 
(4) For any $m\in \mathbb{Z}\cap \partial(x, L_0-4)$ there exists an integer $n$ with $m-R\leq n\leq m$ satisfying 
\[ w_{n,m-n} = 1.\]
\end{lemma}

Before proving the lemma, we explain its intuitive meaning.
The weight $w_n$ is the tax paid by the interval $I(x,n)$.
The entry $w_{nm}$ is the (rescaled) money taken from $I(x,n)$ which is used for the care of the point $n+m\in \mathbb{Z}$.
A point $l\in \mathbb{Z}$ \textit{becomes happy} if there exist $n$ and $m\in [0,R]$ satisfying $n+m=l$ and $w_{nm}=1$.
The condition (2) means that intervals of length $\leq L_1$ do not pay tax.
The condition (3) (roughly) means that the tax taken from $I(x,n)$ does not exceed $|I(x,n)|-L_1$.
The condition (4) means that \textit{we achieve the perfect social welfare}, namely every wild point becomes happy.
($w_{n,m-n}=1$ implies that the point $m$ becomes happy.)

\begin{proof}
Take $x\in Y$.
Set $a_n^{(0)} = (|I(x,n)|-L_1)^+/C$ and $b_n^{(0)} = (L_0-\dist(n,\partial(x)))^+$.
We define $v_{nm}\geq 0$ for $n\in \mathbb{Z}$ and $m\geq 0$
inductively (with respect to $m$) by 
\[ v_{nm} = \min(a_n^{(m)},\, b_{n+m}^{(m)}), \quad 
   a_n^{(m+1)} = a_n^{(m)} -  v_{nm}, \quad b_n^{(m+1)} = b_n^{(m)} - v_{n-m, m}.\]
The heuristic idea behind this process is as follows:
The intervals $I(x,n)$ are donors of tax, and the integral points on the line are receivers.
$a_n^{(0)}$ is the money that the interval $I(x,n)$ can pay as the tax, 
and $b_n^{(0)}$ is the money we need for the care of the point $n$.
At the $m$-th step of the process, the interval $I(x,n)$ pays $v_{nm}$ and we use it for the point $n+m$.
After the $m$-th step, $I(x,n)$ still have the extra money $a_n^{(m+1)}$ and points $n$ still need $b_n^{(m+1)}$.
At each step, $I(x,n)$ pays \textit{as much as possible}.
Namely, if $a_n^{(m)}\geq b_{n+m}^{(m)}$ then $I(x,n)$ pays $b_{n+m}^{(m)}$ and the point $n+m$ becomes satisfied.
If $a_n^{(m)}<b_{n+m}^{(m)}$ then $I(x,n)$ pays $a_n^{(m)}$ and loses all its ability to help integral points.

Every point is satisfied after the $R$-th step:
\[ b_n^{(m)} = v_{nm} = 0 \quad (m>R).\]
This is because the condition $b_n^{(m)}>0$ implies 
\[ \sum_{n-m+1\leq k\leq n} a_k^{(0)} < \sum_{n-m+1\leq k\leq n} b_k^{(0)}, \]
which is impossible for $m=R+1$ by Lemma \ref{lemma: there are enough long intervals}.

We set $v_n = (v_{n0}, \dots, v_{nR})$.
This construction is equivariant: $v_n(Sx) = v_{n+1}(x)$.
Moreover for all integers $n$
\begin{equation}  \label{eq: conditions of v_n}
    C\sum_{m=0}^R v_{nm} \leq (|I(x,n)|-L_1)^+, \quad 
   \sum_{m=0}^R v_{n-m,m} = (L_0-\dist(n,\partial(x))^+.
\end{equation}  
The former inequality holds because $\sum_{m=0}^R v_{nm}$ is the tax paid by $I(x,n)$ and does not exceed $a_n^{(0)}$.
The latter equality holds because every point is satisfied after the $R$-th step.

We choose continuous functions $\alpha:\mathbb{R}\to [1,R]$ and $\beta: \mathbb{R}\to [0,1]$ satisfying 
\[ \alpha(0)=R, \> \alpha(t) = 1 \> (t\geq 1), \quad 
    \beta(t) = 0 \> (t\leq 1), \> \beta(t)=1 \> (t\geq 2).  \]
For $(x_0,x_1,\dots,x_R)\in \mathbb{R}^{R+1}$ we define 
$A(x_0,x_1,\dots,x_R) = (y_0, y_1,\dots,y_R)\in \mathbb{R}^{R+1}$ by 
\begin{equation*}
  \begin{split}
    &y_R = Rx_R  ,\quad  y_{R-1} = \alpha(y_R)x_{R-1}, \quad  y_{R-2} = \alpha(\max(y_R, y_{R-1}))x_{R-2}, \quad \dots, \\
    &y_0 = \alpha(\max(y_R ,\dots, y_1))x_0. 
  \end{split}
\end{equation*}  
Note that this definition implies 
\begin{equation}  \label{eq: property of A}
    \#\{m|\, y_m >1\} \leq 1 + \#\{m|\, x_m>1\}
\end{equation}
Define $B:\mathbb{R}^{R+1}\to [0,1]^{R+1}$ by $B(y_0,\dots, y_R) = (\beta(y_0),\dots,\beta(y_R))$.
We define $w(x) = (w_n)_{n\in \mathbb{Z}}$ by $w_n = B(A(v_n))$. 
We check the required conditions.
The continuity and equivariance are obvious.
The condition (2) follows from the former inequality of (\ref{eq: conditions of v_n}).
This inequality with the help of (\ref{eq: property of A}) also implies the condition (3):
\[ \#\{m|\, w_{nm}>0\} \leq 1 + \#\{m|\, v_{nm}>1\} \leq 1+ \frac{1}{C}(|I(x,n)|-L_1)^+.\]
For the condition (4), take $m \in \mathbb{Z}\cap \partial(x,L_0-4)$. 
Set 
\[ n_0 = \min\{n|\, v_{n,m-n}>0\}.\]
We have $m-R\leq n_0\leq m$.
If $v_{n_0, m-n_0}\geq 2$ then $w_{n_0,m-n_0}=1$.
So we assume $v_{n_0,m-n_0}<2$.
From the latter equality of (\ref{eq: conditions of v_n}) and $\dist(m, \partial(x))\leq L_0-4$, 
\[ \sum_{n=n_0+1}^m v_{n, m-n}  = (L_0-\dist(m, \partial(x)))^+ -v_{n_0,m-n_0} > 2.\]
Since $m-n_0\leq R$, there exists $n_0<n_1 \leq m$ satisfying $v_{n_1,m-n_1} > 2/R$.
The condition $v_{n_0,m-n_0}>0$ with $n_0<n_1$ implies that 
the point $m$ is not satisfied after the $(m-n_1)$-th step and that
the interval $I(x,n_1)$ finishes to pay all its tax at the $(m-n_1)$-th step.
Thus we have $v_{n_1, k}= 0$ for $k>m-n_1$.
Then the definitions of $A$ and $B$ imply $w_{n_1,m-n_1}=1$.
This shows the condition (4).
\end{proof}

\section{Proof of Theorem \ref{thm: technical main theorem}} \label{section: proof of technical main theorem}

In this section we combine all the preparations and prove Theorem \ref{thm: technical main theorem}.
Throughout this section we assume the following. 

\begin{itemize}
     \item   $a<b$ are two real numbers. 
     \item   $(Y,S)$ is a dynamical system having the marker property.
     \item   $\Phi:(X,T)\to (Y,S)$ is an extension with $\mdim(X,T)<b-a$.
\end{itemize}

For the convenience of readers, we restate Theorem \ref{thm: technical main theorem}:
\begin{theorem}[=Theorem \ref{thm: technical main theorem}]
For a dense $G_\delta$ subset of $f\in C_T\left(X,B_1(V[a,b])\right)$ the map 
\[  (f,\Phi): X\to B_1(V[a,b])\times Y, \quad x\mapsto (f(x), \Phi(x)) \]
is an embedding.
\end{theorem}

\subsection{Setting of the proof} \label{subsection: setting}

We fix positive numbers $l,\rho,\tau$ satisfying the following.
\begin{itemize}
  \item $\rho\in \mathbb{Q}$ and $\mdim(X,T)<\rho<b-a$.
  \item $l\in \mathbb{N}$ and $l\rho\in \mathbb{N}$.
  \item $\rho+\tau < b-a$.
\end{itemize}
We use these $l,\rho,\tau$ for the construction of the interpolation function $\varphi_\Lambda$ in (\ref{eq: rapidly decreasing interpolation}).
Let $K=K(l,\rho,\tau)$ be the positive number introduced in (\ref{eq: damping of rapidly decreasing interpolation}).
We denote the distance on $X$ by $d$.
Recall that for a natural number $N$ we defined 
\[ d_N(x,y) = \max_{0\leq n <N} d(T^n x, T^n y).\]

For the proof of Theorem \ref{thm: technical main theorem}, it is enough to prove that the set
\[ \bigcap_{n=1}^\infty \left\{ f\in C_T(X,B_1(V[a,b]))|\, \text{$(f,\Phi)$ is a $(1/n)$-embedding w.r.t. $d$}\right\} \]
is dense and $G_\delta$ in $C_T(X,B_1(V[a,b]))$.
This is obviously $G_\delta$ because ``$(1/n)$-embedding'' is an open condition.
So the task is to prove the next proposition.
Its proof occupies all the rest of the paper.

\begin{proposition}  \label{prop: main proposition}
For any positive number $\delta$ and $f\in C_T(X,B_1(V[a,b]))$, 
there exists $g\in C_T(X, B_1(V[a,b]))$ satisfying the following two conditions.

\noindent 
(1) For all $x\in X$ and $t\in \mathbb{R}$, $|f(x)(t)-g(x)(t)| < \delta$.

\noindent 
(2) $(g, \Phi): X\to B_1(V[a,b])\times Y$ is a $\delta$-embedding with respect to $d$.
\end{proposition}

Fix $\delta>0$ and $f\in C_T(X,B_1(V[a,b]))$.
We can assume $|f(x)(t)|\leq 1-\delta$ for all $x\in X$ and $t\in \mathbb{R}$
by replacing $f$ with $(1-\delta)f$ if necessary.
We choose $\delta'>0$ so that 
if a subset $\Lambda\subset \mathbb{R}$ satisfies 
\[  |\lambda-\lambda'|\geq \frac{1}{\rho}, \quad (\forall \lambda, \lambda'\in \Lambda \text{ with $\lambda\neq \lambda'$}),   \]
then 
\begin{equation}  \label{eq: choice of delta'}
    \delta' \cdot \sum_{\lambda\in \Lambda} \frac{K}{1+|t-\lambda|^2} < \delta \quad (\forall t\in \mathbb{R}).
\end{equation}
We choose $0<\varepsilon <\delta$ so that 
\begin{equation}  \label{eq: choice of varepsilon}
   d(x,y) < \varepsilon \Longrightarrow \forall t\in [0,1]:  |f(x)(t)-f(y)(t)| < \delta'.
\end{equation}

We can find a simplicial complex $Q$ with an $\varepsilon$-embedding $\pi:X\to Q$ with respect to $d$.
Let $CQ = [0,1]\times Q/\{0\}\times Q$ be the cone over $Q$. 
For $(t,x)\in [0,1]\times Q$ we denote its equivalence class by $tx\in CQ$.
We set $* = 0x$ (the vertex of the cone).
The cone $CQ$ will be used for \textit{the care of wild points}.

From $\mdim(X,T) < \rho$ there are an integer $N>1$ with $\rho N\in \mathbb{N}$, 
a simplicial complex $P$ of dimension less than $\rho N$ and an 
$\varepsilon$-embedding $\Pi:X\to P$ with respect to $d_{N}$.
For a natural number $n$ we set 
\[ \Pi_n: X\to P^n, \quad x\mapsto (\Pi(x), \Pi(T^Nx), \dots, \Pi(T^{(n-1)N}x)), \]
which is an $\varepsilon$-embedding with respect to $d_{nN}$.
The space $P^n$ will be used for constructing perturbations over long intervals.
The number $n$ will be chosen so large that \textit{the perturbations can fit intervals of various length}.

We choose natural numbers $C_1,C_2$ and a sequence of integers $2<n_0<n_1<n_2<\dots$ satisfying 
$n_k < C_1 k+C_2$ and 
\begin{equation} \label{eq: conditions of n_k}
   \forall n\geq n_k:  n\dim P + k \dim CQ + 1 \leq (n-1)\rho N.
\end{equation}
Here we have used $\dim P < \rho N$.

We set 
\[ C = C_1 N, \quad L_0 = n_0N+4, \quad L_1 = C_1 N + C_2N + 2N.\]
We apply to $(Y,S)$ the construction of Section \ref{section: marker property and Voronoi tiling} with respect to these 
$C,L_0,L_1$.
Then we get natural numbers 
\[     R>M>L> 4L_1+1 + 4CL_0(4L_0+3),   \]
the tiling $\mathbb{R} = \bigcup_{n\in \mathbb{Z}}I(x,n)$ and 
the weight $w(x) = (w_n)_{n\in \mathbb{Z}}\in ([0,1]^{R+1})^{\mathbb{Z}}$ for each $x\in Y$.

\begin{lemma} \label{lemma: number of nontrivial weights}
Let $x\in Y$ and $n\in \mathbb{Z}$ with $I(x,n) \neq \emptyset$. Set  
\[ I(x,n) = [\alpha, \beta], \quad  r = \left \lceil \frac{\alpha-n}{N}\right\rceil, \quad s = \left \lfloor \frac{\beta-n}{N}\right\rfloor. \]
See Figure \ref{fig: r and s}.
If $s-r > n_0$ then 
\[  \#\{m|\, w_{nm}(x)>0\} \leq \max\{k|\, n_k < s-r\}.\]
\end{lemma}

\begin{proof}
  If $|I(x,n)|\leq L_1$ then $w_n=(0,\dots,0)$ by Lemma \ref{lemma: construction of weight} (2).
So we assume $|I(x,n)|>L_1$.
We have $|I(x,n)| < (s-r)N + 2N$.  By Lemma \ref{lemma: construction of weight} (3), the number of $m$ with $w_{nm}(x)>0$ is 
bounded by 
\[ \left\lfloor 1+ \frac{1}{C}(|I(x,n)| - L_1)\right \rfloor \leq
  \left \lfloor 1+ \frac{(s-r)N+2N-L_1}{C_1 N}\right\rfloor = \left\lfloor \frac{s-r-C_2}{C_1}\right\rfloor.\]
Here we have used $C=C_1 N$ and $L_1 = C_1 N+C_2 N+2N$.
From $n_k < C_1 k+C_2$
\[  \left\lfloor \frac{s-r-C_2}{C_1}\right\rfloor = \max\{k|\, C_1 k+C_2\leq s-r\} \leq 
    \max\{k|\, n_k < s-r\}.\]
\end{proof}

\begin{figure}
    \centering
    \includegraphics[,bb= 0 0 512 160]{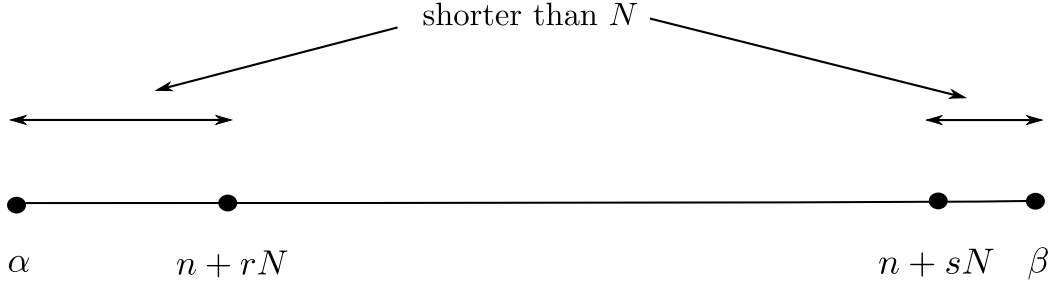}
    \caption{$n+rN$ and $n+sN$ on $I(x,n)$.} \label{fig: r and s}
\end{figure}

We set $W = (CQ)^{R+1} = (CQ)^{\{0,1,2,\dots,R\}}$.
For $0\leq k\leq R$ we define $W_k\subset W$ as the set of $(x_n)_{n=0}^R$ satisfying 
$x_n= *$ except for at most $k$ entries.
Hence $\{(*,\dots,*)\} =W_0 \subset W_1\subset W_2 \subset \dots \subset W_{R+1} = W$.
We have $\dim W_k = k \dim CQ$. 

Consider the disjoint union $P\sqcup CQ$ and take a distance $D$ on it.
We consider $Q = \{1\}\times Q$ as a subspace of $CQ$. So $D$ also gives a distance on $Q$.
There exists $\varepsilon'>0$ such that for $x,y\in X$
\begin{equation} \label{eq: choice of varepsilon'}
   D(\pi(x), \pi(y)) < \varepsilon' \Longrightarrow d(x,y)<\varepsilon,   \quad 
   D(\Pi(x),\Pi(y)) < \varepsilon'  \Longrightarrow d_{N}(x,y) < \varepsilon.
\end{equation}
Let $0\leq m_1 \leq m_2 < n$.
We define a semi-distance $D|_{m_1}^{m_2}$ on $P^n\times W$ by 
\[ D|_{m_1}^{m_2}((x,y), (x',y')) = \max\left(D(x_{m_1}, x'_{m_1}),\dots,D(x_{m_2},x'_{m_2}), D(y_0,y'_0),\dots,D(y_R,y'_R)\right),\]
where 
\[ x= (x_0,\dots,x_{n-1}), x'= (x'_0,\dots,x'_{n-1})\in P^n, \quad 
   y= (y_0,\dots,y_R), y'=(y'_0,\dots,y'_R)\in W.\]
The dependence of $D|_{m_1}^{m_2}$ on $n$ is not explicitly written in this notation.
But we believe that it does not cause a confusion.

\subsection{Successive perturbations}  \label{subsection: successive perturbations}

For a finite set $A$ we define $\mathbb{C}[A]$ as the vector space of all maps from $A$ to $\mathbb{C}$.
This is isomorphic to $\mathbb{C}^{\#A}$.
The following lemma is based on Lemma \ref{lemma: parametric genericity of embedding}.

\begin{lemma}  \label{lemma: successive perturbation}
For all integers $n\geq 1$ and $-M\leq r\leq M$ we can construct simplicial maps 
\[  F_{n,r}:P^n\times W\to \mathbb{C}\left[\left(\frac{1}{\rho}\mathbb{Z}\right)\cap [0,nN)\right]\]
satisfying the following.

\noindent 
(1) For all $x\in X$, $y\in W$ and $t\in (1/\rho)\mathbb{Z}\cap [0,nN)$
\[  |F_{n,r}(\Pi_n(x),y)(t) - f(x)(t)| < \delta'.\]

\noindent 
(2) Let $0\leq k\leq R$ and $n'$ be integers with $n_k\leq n'\leq n$.
For any $-M\leq r<M$ and $0\leq c\leq 1$ the following map is an $\varepsilon'$-embedding with respect to $D|_1^{n'-1}$:
\begin{equation*}
   \begin{split}
   P^{n+1}\times W_k \to  & \mathbb{C}\left[\left(\frac{1}{\rho}\mathbb{Z}\right) \cap [N, n' N)\right] \\
   ((x_0,\dots,x_n),y )\mapsto &  (1-c)F_{n,r}(x_0,\dots,x_{n-1},y)|_{(1/\rho)\mathbb{Z} \cap [N, n' N)} \\
    & + c F_{n,r+1}(x_1,\dots,x_n, y)|_{(1/\rho)\mathbb{Z} \cap [0,(n'-1)N)}.
   \end{split}
\end{equation*}   
The right-hand side is the function whose value of $t\in (1/\rho)\mathbb{Z}\cap [N, n' N)$ is 
\[ (1-c)F_{n,r}(x_0, \dots, x_{n-1},y)(t) + c F_{n,r+1}(x_1,\dots, x_n,y)(t-N).\]
Note that the variables of $F_{n,r+1}$ are $x_1,\dots,x_n,y$ (not $x_0,\dots,x_{n-1},y$).
\end{lemma}

\begin{proof}
First note that the above two conditions (1) and (2) are stable under sufficiently small perturbations of $F_{n,r}$. 
The maps $F_{n,r}$ will be constructed by successive perturbations.
Once the maps satisfy the conditions, their small perturbations also satisfy them.

By Lemma \ref{lemma: approximation by linear map} and
 the choice of $\varepsilon$ in (\ref{eq: choice of varepsilon}), there exists a simplicial map 
\[ F:P\to \mathbb{C}\left[\left(\frac{1}{\rho} \mathbb{Z}\right)  \cap [0,N)\right] \]
 satisfying 
$|F(\Pi(x))(t)-f(x)(t)| < \delta'$ for all $x\in X$ and $t\in (1/\rho)\mathbb{Z}\cap [0,N)$.
For $n<n_0$ we set $F_{n,r}(x,y) = (F(x_0), \dots, F(x_{n-1}))$ for $x=(x_0,\dots,x_{n-1}) \in P^n$ and $y\in W$.
This notation means that 
\[ F_{n,r}(x,y)(t) = F(x_i)(t-iN), \quad (0\leq i < n,\,  t\in (1/\rho)\mathbb{Z}\cap [iN, (i+1)N)).  \]
We will use similar notations below.
These $F_{n,r}$ satisfy the required conditions since the condition (2) is empty for $n<n_0$.
So we assume $n\geq n_0$ and that we have constructed $F_{n-1, r}$ for all $-M\leq r\leq M$.
We try to construct $F_{n,r}$.

Consider 
\begin{equation*}
   \begin{split}
    (F_{n-1,r},F):P^n\times W  &\to \mathbb{C}\left[\left(\frac{1}{\rho}\mathbb{Z}\right)\cap [0,nN)\right],  \\
    (x_0,\dots,x_{n-1},y)  &\mapsto (F_{n-1,r}(x_0,\dots,x_{n-2},y), F(x_{n-1})).
   \end{split}
\end{equation*}   
These satisfy the condition (1) and also the condition (2) for $n_k\leq n'\leq n-1$.
So we will construct $F_{n,r}$ by slightly perturbing $(F_{n-1,r},F)$.
Consider the following condition:

(3) Take integers $-M\leq r<M$, $0\leq k\leq R$ with $n_k\leq n$ and a real number $0\leq c\leq 1$.
The following map is an $\varepsilon'$-embedding with respect to $D|_1^{n-1}$.
\begin{equation*}
  \begin{split}
  P^n \times W_k \to & \mathbb{C}\left[\left(\frac{1}{\rho}\mathbb{Z}\right) \cap [N, n N)\right] \\
  ((x_0, \dots,x_{n-1}), y) \mapsto & (1-c)F_{n,r}(x_0,\dots,x_{n-1}, y)|_{(1/\rho)\mathbb{Z}\cap [N,n N)} \\
   & + c F_{n-1,r+1}(x_1,\dots,x_{n-1},y)|_{(1/\rho)\mathbb{Z}\cap [0,(n-1)N)}.
  \end{split}
\end{equation*}
The main difference between the conditions (2) and (3) is that $F_{n,r+1}(x_1,\dots,x_n,y)$ in (2) is replaced with 
$F_{n-1,r+1}(x_1,\dots,x_{n-1},y)$ in (3).

Note that the real dimension of $\mathbb{C}[(1/\rho)\mathbb{Z}\cap [N,nN)]$ is $2(n-1)\rho N$.
By Corollary \ref{cor: embeddings are dense}
and the choice of $n_k$ in (\ref{eq: conditions of n_k}), we can assume that the condition (3) is satisfied for $c=1$
after replacing the maps $F_{n-1,r+1}$ by small perturbations (if necessary).

By using Lemma \ref{lemma: parametric genericity of embedding} and (\ref{eq: conditions of n_k}), 
we can construct $F_{n,-M}$ as a small perturbation of 
$(F_{n-1,-M}, F)$ so that it satisfies the condition (3) for $r=-M$.
Then, if $F_{n,-M+1}$ is a sufficiently small perturbation of $(F_{n-1,-M+1}, F)$, the condition (2) is satisfied for $r =-M$.
Moreover we can assume that it satisfies the condition (3) for $r=-M+1$
by the same reason.
By continuing this process, we can construct $F_{n,r}$ inductively (with respect to $r$) so that they satisfy the required properties.
\end{proof}

For $-M\leq r\leq M$ we set 
\[  G_r = F_{M,r}:P^M\times W \to \mathbb{C}\left[\left(\frac{1}{\rho}\mathbb{Z}\right)\cap [0, M N)\right]. \]
Indeed any $F_{n,r}$ will do the same work if $n$ is sufficiently large.
We use the choice $F_{M,r}$ because $|I(x,n)|<M$ by Lemma \ref{lemma: properties of tiling} (1).

\subsection{Construction of the map $g$} \label{subsection: construction of the map g}

Take $x\in X$. We will define $g(x)\in B_1(V[a,b])$.
We define $E(x)$ as the set of integers $n$ with $I(\Phi(x),n) \neq \emptyset$.
Take $n\in E(x)$. We set 
\[ I(\Phi(x),n) = [\alpha_{x,n}, \beta_{x,n}], \quad r_{x,n} = \left\lceil \frac{\alpha_{x,n}-n}{N}\right\rceil, \quad 
   s_{x,n} = \left\lfloor \frac{\beta_{x,n}-n}{N}\right\rfloor.  \]
We define $0\leq c_{x,n}, c'_{x,n}< 1$ by 
\[ c_{x,n} = \frac{n+r_{x,n}N-\alpha_{x,n}}{N}, \quad c'_{x,n} = \frac{\beta_{x,n}-n-s_{x,n}N}{N}.\]
See Figure \ref{fig: I(Phi(x),n) and several points on it}.

\begin{figure}
    \centering
    \includegraphics[,bb= 0 0 530 135]{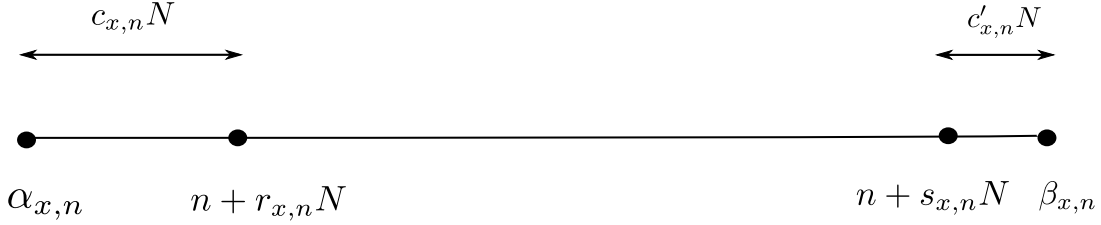}
    \caption{$I(\Phi(x),n)$ and several points on it.} \label{fig: I(Phi(x),n) and several points on it}
\end{figure}

Let $\boldsymbol{\delta}_0$ and $\boldsymbol{\delta}_1$ be the delta measures on the two-points space $\{0,1\}$
concentrated at $0$ and $1$ respectively.
We define a probability measure on $\{0,1\}\times \{0,1\}$ by 
\[ \mu_{x,n} = \left(c_{x,n} \boldsymbol{\delta}_0 + (1-c_{x,n})\boldsymbol{\delta}_1\right) 
    \times \left(c'_{x,n}\boldsymbol{\delta}_0 + (1-c'_{x,n})\boldsymbol{\delta}_1\right).   \]
We define a probability measure on $\prod_{n\in E(x)} \{0,1\}^2$ by 
\[ \mu_x = \prod_{n\in E(x)} \mu_{x,n}. \]

Take 
\[ \boldsymbol{\theta} \in \prod_{n\in E(x)} \{0,1\}^2, \quad \boldsymbol{\theta} = ((\theta_n, \theta'_n))_{n\in E(x)}, \quad 
  \theta_n\in \{0,1\}, \> \theta'_n\in \{0,1\}.\]
For $n\in E(x)$ we define 
\[ \Lambda(x, \boldsymbol{\theta}, n) = n + \left( \left(\frac{1}{\rho}\mathbb{Z}\right) \cap [(r_{x,n}+\theta_n)N, (s_{x,n}-\theta'_n)N)\right) 
    \subset I(\Phi(x),n).  \]
When $r_{x,n}+\theta_n \geq s_{x,n}-\theta'_n$, this is empty.
See Figure \ref{fig: Lambda(x,theta,n)}.
\begin{figure}
    \centering
    \includegraphics[,bb= 0 0 530 133]{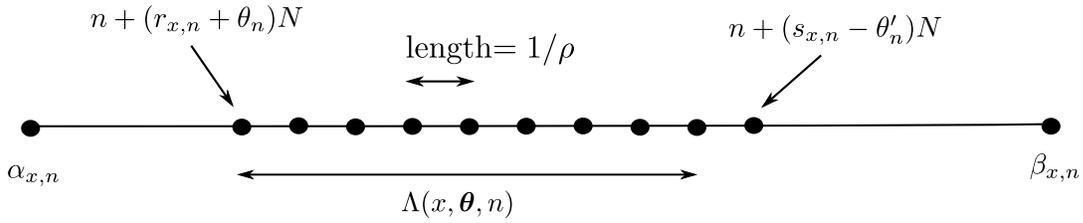}
    \caption{$\Lambda(x,\boldsymbol{\theta},n)$ on $I(\Phi(x),n)$.} \label{fig: Lambda(x,theta,n)}
\end{figure}
We define a subset of $\mathbb{R}$ by 
\[  \Lambda(x, \boldsymbol{\theta}) = \bigcup_{n\in E(x)} \Lambda(x, \boldsymbol{\theta}, n). \]
The distance between any two distinct points of $\Lambda(x, \boldsymbol{\theta})$ is $\geq 1/\rho$.
So for any $\lambda\in \Lambda(x, \boldsymbol{\theta})$ 
the set $-\lambda+ \Lambda(x, \boldsymbol{\theta})$ satisfies Conditions \ref{condition: interpolation} (1) and (2).
Let $\varphi_{-\lambda+\Lambda(x, \boldsymbol{\theta})} \in V[-(\rho+\tau)/2,(\rho+\tau)/2]$ 
be the interpolation function introduced in (\ref{eq: rapidly decreasing interpolation}).
We define $\varphi_{x, \boldsymbol{\theta}, \lambda}$ by 
\[ \varphi_{x, \boldsymbol{\theta}, \lambda}(t) 
 = \exp\left(2\pi \sqrt{-1}\frac{a+b}{2} (t-\lambda) \right) \varphi_{-\lambda+\Lambda(x, \boldsymbol{\theta})}(t-\lambda).\]
This satisfies 
\begin{itemize}
  \item $\varphi_{x, \boldsymbol{\theta}, \lambda}\in V[a,b]$ because $\rho+\tau<b-a$.
  \item $\varphi_{x, \boldsymbol{\theta},\lambda}(\lambda)=1$ and $\varphi_{x,\boldsymbol{\theta},\lambda}(\lambda')=0$
   for all $\lambda'\in \Lambda(x, \boldsymbol{\theta})\setminus \{\lambda\}$.
   \item $\varphi_{x, \boldsymbol{\theta},\lambda}$ is rapidly decreasing and 
   \begin{equation}  \label{eq: damping estimate}
      |\varphi_{x, \boldsymbol{\theta},\lambda}(t)| \leq \frac{K}{1+|t-\lambda|^2}.  
   \end{equation}
\end{itemize}

Let $w(\Phi(x)) = (w_n)_{n\in \mathbb{Z}}$, $w_n = (w_{n0},\dots, w_{nR}) \in [0,1]^{R+1}$, be the weight introduced in 
Lemma \ref{lemma: construction of weight}.
Let $n\in E(x)$. We set 
\[  y_{x,n} = (w_{n0} \pi(T^n x), w_{n1}\pi(T^{n+1}x), \dots,w_{nR}\pi(T^{n+R}x))\in W = (CQ)^{R+1}. \]
For $\lambda\in \Lambda(x, \boldsymbol{\theta}, n)$ we set 
\[ u(x, \boldsymbol{\theta},n,\lambda) 
   = G_{r_{x,n}+\theta_n}\left(\Pi_M(T^{n+(r_{x,n}+\theta_n)N}x),y_{x,n}\right)\left(\lambda-n-(r_{x,n}+\theta_n)N\right)
   - f(x)(\lambda).\]
Note 
\[  f(x)(\lambda) = f(T^{n+(r_{x,n}+\theta_n)N}x)\left(\lambda - n-(r_{x,n}+\theta_n)N\right).  \]
Hence by Lemma \ref{lemma: successive perturbation} (1)
\begin{equation}  \label{eq: perturbation term is small}
    |u(x, \boldsymbol{\theta},n,\lambda)| <\delta'. 
\end{equation}

We define a function $g(x, \boldsymbol{\theta})$ in $V[a,b]$ by 
\[ g(x, \boldsymbol{\theta})(t) = f(x)(t) + \sum_{n\in E(x)} \sum_{\lambda\in \Lambda(x, \boldsymbol{\theta},n)} 
   u(x, \boldsymbol{\theta}, n,\lambda) \varphi_{x, \boldsymbol{\theta}, \lambda}(t).\]
From (\ref{eq: choice of delta'}), (\ref{eq: damping estimate}) and (\ref{eq: perturbation term is small})
\[  |g(x, \boldsymbol{\theta})(t)-f(x)(t)| < \delta.\]
Finally we define $g(x)\in V[a,b]$ by 
\[ g(x) = \int_{\prod_{n\in E(x)} \{0,1\}^2} g(x, \boldsymbol{\theta})\,  d\mu_x(\boldsymbol{\theta}). \]
This satisfies $|g(x)(t)-f(x)(t)|<\delta$.
Since $|f(x)(t)|\leq 1-\delta$, we have $g(x)\in B_1(V[a,b])$.

For every $n\in E(x)$ with $r_{x,n}+1 < s_{x,n}-1$ 
\begin{equation}  \label{eq: formula for g}
   \begin{split}
    g(x)&|_{n+ \left((1/\rho)\mathbb{Z}\cap [(r_{x,n}+1)N, (s_{x,n}-1)N)\right)} \\
    = & \,  c_{x,n} G_{r_{x,n}}\left(\Pi_M(T^{n+r_{x,n}N}x),  y_{x,n}\right)|_{(1/\rho)\mathbb{Z}\cap [N,(s_{x,n}-r_{x,n}-1)N)} \\
   &  + (1-c_{x,n}) G_{r_{x,n}+1}\left(\Pi_M(T^{n+(r_{x,n}+1)N}x),y_{x,n}\right)|_{(1/\rho)\mathbb{Z}\cap [0,(s_{x,n}-r_{x,n}-2)N)}.
   \end{split}
\end{equation}

\begin{lemma}  
The map 
\[ X\ni x\mapsto g(x)\in B_1(V[a,b]) \]
is equivariant and continuous.
\end{lemma}
\begin{proof}
The check of the equivariance is direct.
We have $I(\Phi(Tx),n) = -1+I(\Phi(x),n+1)$.
Hence
$E(Tx) = -1+E(x)$ and for $n\in E(Tx)$
\[ r_{Tx,n} = r_{x,n+1}, \quad s_{Tx,n} = s_{x,n+1}, \quad c_{Tx,n} = c_{x,n+1},\quad c'_{Tx,n} = c'_{x,n+1}.\]
 We have a one to one correspondence between $\prod_{n\in E(x)}\{0,1\}^2$ and 
$\prod_{n\in E(Tx)} \{0,1\}^2$ by 
\[ \boldsymbol{\theta} \longleftrightarrow \tilde{\boldsymbol{\theta}}, \quad (\tilde{\theta}_n,\tilde{\theta}'_n) = (\theta_{n+1}, \theta'_{n+1}).\]
Under this identification, we have $\mu_{Tx} = \mu_x$.
We can check the following.
\[ \Lambda(Tx,\tilde{\boldsymbol{\theta}},n) = -1+ \Lambda(x, \boldsymbol{\theta}, n+1), 
    \quad \Lambda(Tx,\tilde{\boldsymbol{\theta}}) = -1+\Lambda(x, \boldsymbol{\theta}),\quad 
    \varphi_{Tx,\tilde{\boldsymbol{\theta}},\lambda}(t) = \varphi_{x,\boldsymbol{\theta},\lambda+1}(t+1), \]
\[ y_{Tx,n} = y_{x,n+1} \quad \text{by Lemma \ref{lemma: construction of weight} (1)}, 
   \quad u(Tx,\tilde{\boldsymbol{\theta}},n,\lambda) = u(x, \boldsymbol{\theta}, n+1,\lambda+1).\]
Then 
\[ g(Tx,\tilde{\boldsymbol{\theta}})(t) = g(x, \boldsymbol{\theta})(t+1), \quad 
   g(Tx)(t) = g(x)(t+1).\]

The proof of the continuity is slightly involved. Let $x\in X$.
Discontinuity appears in the two places of the above construction.
\begin{itemize}
   \item If $I(\Phi(x),n)$ is one point, then it may become empty after $x$ moves slightly.
   \item The integers $r_{x,n}$ and $s_{x,n}$ may jump when $c_{x,n}=0$ or $c'_{x,n}=0$.
\end{itemize}
The first issue causes no problem because $\Lambda(x, \boldsymbol{\theta},n)$ is empty and does not contribute to the value of $g(x)$ 
if $|I(\Phi(x),n)|=0$.
The second issue is more serious and causes a problem that 
$g(x, \boldsymbol{\theta})$ does \textit{not} depend continuously on $x$.
We introduced the probability measure $\mu_{x,n}$ for dealing with this problem.
Let $\mathcal{C}$ and $\mathcal{C}'$ be the sets of integers $n\in E(x)$ satisfying 
$c_{x,n}=0$ and $c'_{x,n}=0$ respectively.
These are the positions where the difficulty occurs.

Let $A$ and $\eta$ be positive numbers. 
Suppose $x'\in X$ is sufficiently close to $x$.
We want to show $|g(x')(t)-g(x)(t)| < \eta$ for $|t|\leq A$.
Let $B>0$ be a sufficiently large number.
We can assume $E(x')\cap [-A-B,A+B] \subset E(x)\cap [-A-B,A+B]$ and that 
every integer $n$ in the difference of these two sets satisfies $|I(\Phi(x),n)|=0$.
This means that these two sets are essentially equal.

Take $\boldsymbol{\theta} \in \prod_{n\in E(x')} \{0,1\}^2$.
We define $\boldsymbol{\Theta}(x', \boldsymbol{\theta})\in \prod_{n\in E(x')}\{0,1\}^2$ as follows.
\begin{itemize}
  \item For $|n|>A+B$ we set $(\Theta(x', \boldsymbol{\theta})_n,\Theta(x', \boldsymbol{\theta})'_n) = (0,0)$.
  \item Let $|n|\leq A+B$. 
  If $n\not\in \mathcal{C}$ then $\Theta(x', \boldsymbol{\theta})_n= \theta_n$. 
  If $n \not \in \mathcal{C}'$ then $\Theta(x', \boldsymbol{\theta})'_n = \theta'_n$.
  \item For $n\in \mathcal{C}\cap [-A-B,A+B]$, we define $\Theta(x',\boldsymbol{\theta})_n\in \{0,1\}$ by 
  \[  r_{x',n}+\Theta(x', \boldsymbol{\theta})_n = r_{x,n} + 1.  \]
  \item For $n\in \mathcal{C}'\cap [-A-B,A+B]$ we define $\Theta(x', \boldsymbol{\theta})'_n \in \{0,1\}$ by 
  \[  s_{x',n} - \Theta(x', \boldsymbol{\theta})'_n = s_{x,n} - 1. \]
\end{itemize}
Note that for $n\in \mathcal{C}\cap [-A-B,A+B]$ the number $c_{x',n}$ is very close to $0$ or $1$, and that
the measure $c_{x',n}\boldsymbol{\delta}_0 + (1-c_{x',n}) \boldsymbol{\delta}_1$ 
is almost equal to the delta measure at $\Theta(x', \boldsymbol{\theta})_n$. 
Similarly for $n\in \mathcal{C}' \cap [-A-B,A+B]$.
Then we can assume that for $|t|\leq A$
\[  \left|\int_{\prod_{E(x')} \{0,1\}^2} g(x', \boldsymbol{\theta})(t) d\mu_{x'}(\boldsymbol{\theta}) 
     - \int_{\prod_{E(x')} \{0,1\}^2} g(x',\boldsymbol{\Theta}(x',\boldsymbol{\theta}))(t) d\mu_{x'}(\boldsymbol{\theta}) \right| < \eta/3. \]
Here we have also used Lemma \ref{lemma: continuity of rapidly decreasing interpolation}
(with the assumption $B\gg 1$) and (\ref{eq: damping estimate}).
Since the two sets $E(x)\cap [-A-B,A+B]$ and $E(x')\cap [-A-B,A+B]$ are essentially equal, applying 
 Lemma \ref{lemma: continuity of rapidly decreasing interpolation} and (\ref{eq: damping estimate}) again, we get
\[ \left| \int_{\prod_{E(x')} \{0,1\}^2} g(x', \boldsymbol{\Theta}(x', \boldsymbol{\theta}))(t) d\mu_{x'}(\boldsymbol{\theta}) 
    -  \int_{\prod_{E(x)} \{0,1\}^2} g(x,\boldsymbol{\Theta}(x,\boldsymbol{\theta}))(t) d\mu_{x}(\boldsymbol{\theta}) \right|  < \eta/3 \]
for $|t|\leq A$.    
Thus $|g(x')(t) -g(x)(t)|< \eta$ for $|t|\leq A$.
\end{proof}

The rest of the task is to show that the map 
\[ (g,\Phi): X\to B_1(V[a,b])\times Y, \quad x\mapsto (g(x),\Phi(x)) \]
is a $\delta$-embedding with respect to $d$.
Suppose $x,x'\in X$ satisfy $(g(x),\Phi(x)) = (g(x'),\Phi(x'))$.
We want to show $d(x,x')<\delta$.
Let $w(\Phi(x)) = w(\Phi(x')) = (w_n)_{n\in \mathbb{Z}}$.
We divide the argument into two cases, according to whether the origin is tame or wild.

\textbf{Case 1.}
Suppose $\dist(0,\partial(\Phi(x))) >L_0-4=n_0 N >2N$.
Take an integer $n$ with $0\in I(\Phi(x),n)$.
Then $|I(\Phi(x),n)| > 2n_0 N$ and hence $s_{x,n}-r_{x,n}>n_0$.
Let $k$ be the maximum integer satisfying $n_k < s_{x,n}-r_{x,n}$.
By Lemma \ref{lemma: number of nontrivial weights}
the points $y_{x,n}$ and $y_{x',n}$ belong to $W_k$.
Then by (\ref{eq: formula for g}) and Lemma \ref{lemma: successive perturbation} (2) (with $n'= s_{x,n}-r_{x,n}-1  \geq n_k$), 
we get 
\[ D_1^{s_{x,n}-r_{x,n}-2}\left(\left(\Pi_M(T^{n+r_{x,n}N}x), y_{x,n}\right), \left(\Pi_M(T^{n+r_{x,n}N}x'), y_{x',n}\right)\right) 
   < \varepsilon'. \]
From the second condition on $\varepsilon'$ in (\ref{eq: choice of varepsilon'}), 
this implies that for all integers $i$ with $n+(r_{x,n}+1)N\leq i < n+(s_{x,n}-1)N$ 
\[  d(T^i x, T^i x') < \varepsilon.  \]
Since $\dist(0,\partial I(\Phi(x),n)) > n_0N > 2N$, the origin is contained in $[n+(r_{x,n}+1)N, n+(s_{x,n}-1)N)$.
Thus we get $d(x,x') <\varepsilon <\delta$.
Note that the points $y_{x,n}$ and $y_{x',n}$ do not play any role in this argument.
They will become crucial in Case 2.

\textbf{Case 2.} 
Suppose $\dist(0,\partial(\Phi(x)))\leq L_0-4$.
By Lemma \ref{lemma: construction of weight} (4) there exists an integer $n\in [-R,0]$ with $w_{n,-n}=1$.
By Lemma \ref{lemma: construction of weight} (2) this implies $|I(\Phi(x),n)| >L_1 > C_2 N + 2N$.
Then $s_{x,n}-r_{x,n} > C_2 > n_0$.
Let $k$ be the maximum integer satisfying $n_k< s_{x,n}-r_{x,n}$.
By Lemma \ref{lemma: number of nontrivial weights}, $y_{x,n}, y_{x',n}\in W_k$.
By (\ref{eq: formula for g}) and Lemma \ref{lemma: successive perturbation} (2) 
\[   D_1^{s_{x,n}-r_{x,n}-2}\left(\left(\Pi_M(T^{n+r_{x,n}N}x), y_{x,n}\right), \left(\Pi_M(T^{n+r_{x,n}N}x'), y_{x',n}\right)\right) 
   < \varepsilon'. \]
This is the same as in Case 1.
But the next step is different.
Since $w_{n,-n}=1$, the points $\pi(x) = w_{n,-n}\pi(x)$ and $\pi(x') = w_{n,-n}\pi(x')$
appear as the $n$-th entries of $y_{x,n}$ and $y_{x',n}$ respectively.
Therefore
\[ D(\pi(x),\pi(x')) \leq  D_1^{s_{x,n}-r_{x,n}-2}\left(\left(\Pi_M(T^{n+r_{x,n}N}x), y_{x,n}\right), \left(\Pi_M(T^{n+r_{x,n}N}x'), y_{x',n}\right)\right) 
   < \varepsilon'. \]
By the first condition on $\varepsilon'$ in (\ref{eq: choice of varepsilon'}), we finally get $d(x,x') < \varepsilon <\delta$.

We have completed the proof of Proposition \ref{prop: main proposition}.
Thus Theorem \ref{thm: technical main theorem} has been proved.

\vspace{0.5cm}

\address{Yonatan Gutman, Institute of Mathematics, Polish Academy of Sciences,
ul. \'{S}niadeckich~8, 00-656 Warszawa, Poland}

\textit{E-mail address}: \texttt{y.gutman@impan.pl}

\vspace{0.5cm}

\address{ Masaki Tsukamoto \endgraf
Department of Mathematics, Kyoto University, Kyoto 606-8502, Japan}

\textit{E-mail address}: \texttt{tukamoto@math.kyoto-u.ac.jp}


\begin{thebibliography}{99}



\bibitem[Ahl]{Ahlfors}
L.V. Ahlfors, 
Complex Analysis, 
Third edition, McGraw-Hill, New York, 1978.




\bibitem[Aus]{Auslander}
J. Auslander,
Minimal flows and their extensions,
North-Holland, Amsterdam, 1988.





\bibitem[Beu]{Beurling}
Collected works of Arne Beurling, Vol. 2, Harmonic Analysis,
Edited by L. Carleson, P. Malliavan, J. Neuberger, and J. Wermer,  Birkh\"{a}user, Boston-Basel-Berlin, 1989.



\bibitem[BCR]{Bochnak--Coste--Roy}
J. Bochnak, M. Coste, M-F. Roy, 
Real algebraic geometry,
Translated from the 1987 French original, revised by the authors,
Springer-Verlag, Berlin, 1998.



\bibitem[DM]{Dym--McKean}
H. Dym, H.P. McKean, 
Fourier series and integrals, 
Academic Press, New York-London, 1972.





\bibitem[Gro]{Gromov}
M. Gromov, 
Topological invariants of dynamical systems and spaces of holomorphic maps: I,
Math. Phys. Anal. Geom. \textbf{2} (1999) 323-415.



\bibitem[Gut11]{Gutman 1}
Y. Gutman,
Embedding $\mathbb{Z}^k$-actions in cubical shifts and 
$\mathbb{Z}^k$-symbolic extensions,
Ergodic Theory Dynam. Systems \textbf{31} (2011) 383-403.


\bibitem[Gut12]{Gutman 2}
Y. Gutman, Mean dimension and Jaworski-type theorems, 
Proceedings of the London Mathematical Society \textbf{111(4)} (2015) 831-850.



\bibitem[GLT]{Gutman--Lindenstrauss--Tsukamoto}
Y. Gutman, E. Lindenstrauss, M. Tsukamoto, 
Mean dimension of $\mathbb{Z}^k$-actions, 
preprint, arXiv:1510.01605.



\bibitem[GT]{Gutman--Tsukamoto}
Y. Gutman, M. Tsukamoto,
Mean dimension and a sharp embedding theorem: extensions of aperiodic subshifts,
Ergodic Theory Dynam. Systems \textbf{34} (2014) 1888-1896.




\bibitem[Hay]{Hayman}
W.K. Hayman, 
Meromorphic functions, 
Clarendon Press, Oxford, 1964.





\bibitem[Jaw]{Jaworski}
A. Jaworski, Ph.D. Thesis, University of Maryland (1974).




\bibitem[Lig03]{Lightwood 1}
S.J. Lightwood,  Morphisms from non-periodic $\mathbb{Z}^2$-subshifts. I. Constructing
 embeddings from homomorphisms,
 Ergodic Theory Dynam. Systems  \textbf{23} (2003) 587--609.
		


\bibitem[Lig04]{Lightwood 2}
S.J. Lightwood,  Morphisms from non-periodic $\mathbb{Z}^2$ subshifts. II. Constructing
 homomorphisms to square-filling mixing shifts of finite type.
 Ergodic Theory Dynam. Systems  \textbf{24}  (2004) 1227--1260.



\bibitem[Lin]{Lindenstrauss}
E. Lindenstrauss,
Mean dimension, small entropy factors and an embedding theorem,
Inst. Hautes \'{E}tudes Sci. Publ. Math. \textbf{89} (1999) 227-262.





\bibitem[LT]{Lindenstrauss--Tsukamoto}
E. Lindenstrauss, M. Tsukamoto,
Mean dimension and an embedding problem: an example,
Israel J. Math. \textbf{199} (2014) 573-584.






\bibitem[LW]{Lindenstrauss--Weiss}
E. Lindenstrauss, B. Weiss,
Mean topological dimension,
Israel J. Math. \textbf{115} (2000) 1-24.






\bibitem[Nik]{Nikol'skii}
S.M. Nikol'skii,
Approximation of functions of several variables and 
imbedding theorems,
Translated from the Russian by John M. Danskin, Jr, 
Springer-Verlag, New York-Heidelberg, 1975.





\bibitem[NW]{Noguchi--Winkelmann}
J. Noguchi, J. Winkelmann, 
Nevanlinna theory in several complex variables and Diophantine approximation, 
Springer, Tokyo, 2014.



\bibitem[Sch]{Schwartz}
L. Schwartz, 
Th\'{e}orie des distributions, 
Hermann, Paris, 1966.




\end{thebibliography}
\end{document}